\documentclass[10pt]{amsart}

\usepackage{amssymb,amsmath,amsthm,color}
\usepackage[all]{xy}
\usepackage{mathrsfs}
\usepackage[OT2,OT1]{fontenc}
\newcommand\cyr{%
\renewcommand\rmdefault{wncyr}%
\renewcommand\sfdefault{wncyss}%
\renewcommand\encodingdefault{OT2}%
\normalfont \selectfont} \DeclareTextFontCommand{\textcyr}{\cyr}

\newtheorem{theorem}{Theorem}[section]
\newtheorem{proposition}[theorem]{Proposition}
\newtheorem{lemma}[theorem]{Lemma}
\newtheorem{corollary}[theorem]{Corollary}
\theoremstyle{definition} 
\newtheorem{definition}[theorem]{Definition}

\DeclareMathOperator{\SL}{SL}
\DeclareMathOperator{\Sp}{Sp}\DeclareMathOperator{\M}{M}
\DeclareMathOperator{\NM}{NM}
\DeclareMathOperator{\GL}{GL}\DeclareMathOperator{\GSp}{GSp}
\DeclareMathOperator{\modu}{mod}\DeclareMathOperator{\cusp}{cusp}
\DeclareMathOperator{\spin}{spin}
\DeclareMathOperator{\SGN}{sgn}\DeclareMathOperator{\Tr}{Tr}
\DeclareMathOperator{\diag}{diag}\DeclareMathOperator{\st}{st}
\DeclareMathOperator{\J}{J}
\DeclareMathOperator{\IM}{Im}
\DeclareMathOperator{\Real}{Re}\DeclareMathOperator{\ord}{ord}
\DeclareMathOperator{\alg}{alg}\DeclareMathOperator{\Gal}{Gal}
\DeclareMathOperator{\Frob}{Frob}\DeclareMathOperator{\Rad}{Rad}
\DeclareMathOperator{\coh}{H}\DeclareMathOperator{\cont}{cont}\DeclareMathOperator{\cris}{cris}
\DeclareMathOperator{\Cris}{Cris}\DeclareMathOperator{\Fil}{Fil}
\DeclareMathOperator{\ur}{ur}\DeclareMathOperator{\Sel}{Sel}
\newcommand{\inte}{\mathbb{Z}}
\newcommand{\intep}{\inte_{p}}

\newcommand{\rat}{\mathbb{Q}}
\newcommand{\real}{\mathbb{R}}
\newcommand{\cmplx}{\mathbb{C}}

\newcommand{\F}{\mathbb{F}}
\newcommand{\T}{\mathbb{T}}
\newcommand{\A}{\mathbb{A}}

\newcommand{\Ff}{F_{f}}
\newcommand{\p}{\mathfrak{p}}

\newcommand{\Gtwo}{\Gamma_{2}}
\renewcommand{\i}{{\bf i}}
\newcommand{\ROI}{\mathcal{O}}
\newcommand{\h}[1]{\mathfrak{h}^{#1}}

\newcommand{\Slash}[1]{|_{#1}}
\newcommand{\mat}[4]{\begin{pmatrix}{#1}&{#2}\\{#3}&{#4}\end{pmatrix}}

\newcommand{\mc}[1]{\mathscr{#1}}
\DeclareMathOperator{\ev}{ev}\DeclareMathOperator{\Hom}{Hom}\DeclareMathOperator{\Res}{Res}
\newcommand{\mf}[1]{\mathfrak{#1}}
\newcommand{\ov}[1]{\overline{#1}}
\DeclareMathOperator{\semi}{ss}\DeclareMathOperator{\G}{G}
\newcommand{\lat}{\mathscr{L}}\DeclareMathOperator{\Id}{Id}\DeclareMathOperator{\End}{End}\DeclareMathOperator{\Ann}{Ann}\DeclareMathOperator{\Ga}{G}
\DeclareMathOperator{\id}{id}\DeclareMathOperator{\Ext}{Ext}\DeclareMathOperator{\Fitt}{Fitt}\DeclareMathOperator{\Aut}{Aut}
\newcommand{\absgal}{\G_{\rat}}\DeclareMathOperator{\length}{length}\newcommand{\gal}[1]{\Ga_{#1}}

\begin{document}

\title[Cuspidality of Pullbacks of Siegel Eisenstein series]{On the cuspidality of pullbacks of Siegel Eisenstein series and applications to the Bloch-Kato conjecture}
\author{Jim Brown}
\address{Department of Mathematics\\
California Institute of Technology\\ Pasadena, CA 91125}
\email{jimlb@caltech.edu}

\begin{abstract}  
Let $k > 3$ be an integer and $p$ a prime with $p > 2k-2$. Let $f$ be a newform of weight $2k-2$ and level 1 so that $f$ is ordinary at $p$ and $\ov{\rho}_{f}$ is irreducible.  Under some additional hypotheses we prove that $\ord_{p}(L_{\alg}(k,f)) \leq \ord_{p}(\# S)$ where $S$ is the Pontryagin dual of the Selmer group associated to $\rho_{f} \otimes \varepsilon^{1-k}$ with $\varepsilon$ the $p$-adic cyclotomic character.  We accomplish this by first constructing a congruence between the Saito-Kurokawa lift of $f$ and a non-CAP Siegel cusp form.  Once this congruence is established, we use Galois representations to obtain the lower bound on the Selmer group.
\end{abstract}

\maketitle

\section{Introduction}

The conjecture of Bloch and Kato, also known as the Tamagawa number conjecture, is one of the central outstanding conjectures in 
number theory.  Let $f$ be a newform of weight $2k-2$ and level 1.  The Bloch-Kato conjecture for modular forms roughly states that the special values of the $L$-function associated to $f$ should measure the size of the Selmer group associated to twists of the Galois representation associated to $f$.  In previous work we showed that under suitable hypotheses that one has that if $\ord_{p}(L_{\alg}(k,f)) \geq 1$, then $\ord_{p}(\#\Sel(W)) \geq 1$ (\cite{jimbrown}).  Unfortunately, due to limitations of the method used one was unable to gain any more information then this.  In this paper we improve upon these results yielding the much stronger result that under the same hypotheses needed before roughly speaking one has
\begin{equation*}
\ord_{p}(L_{\alg}(k,f)) \leq \ord_{p}(\# \Sel(W)).
\end{equation*}
The method of proof is in the same spirit as used in \cite{jimbrown} with significant improvements at a couple of steps.  Nevertheless, we briefly recount the method here.  

The general outline of the method goes back to work of Ribet on the proof of the converse of Herbrand's theorem (\cite{ribet}).  This method has been generalized and applied in other contexts by several authors (\cite{klosin}, \cite{skinnerurban}, \cite{wiles}, etc.)   The method employed by Ribet (in a slightly more general form) is as follows. Given a positive integer $k$ and a primitive Dirichlet character $\chi$ 
of conductor $N$ so that $\chi(-1) = (-1)^{k}$, one has an associated Eisenstein series $E_{k,\chi}$ with constant term $\frac{L(1-k,\chi)}{2}$.  For an odd prime $p$ with $p \mid L(1-k,\chi)$ and $p \nmid N$, one can show that there is a cuspidal eigenform $g$ of weight $k$ and level $M$ with $N \mid M$ so that 
$g \equiv E_{k,\chi} (\modu \p)$ for some prime $\p \mid p$.  This congruence is used to 
study the residual Galois representation of $g$.  It is shown that $\overline{\rho}_{g,\p} \simeq \mat{1}{*}{0}{\chi \omega^{k-1}}$ is non-split
where $\omega$ is the reduction of the $p$-adic cyclotomic character.  This allows one to show that $*$ gives a non-zero cohomology class
in $\coh_{\ur}^1(\rat, \chi^{-1}\omega^{1-k})$.   

For our purposes, the character in Ribet's method will be replaced with a newform $f$ of weight $2k-2$ and level 1.  
Associated to $f$ we have its
Saito-Kurokawa lift $\Ff$, our replacement for the Eisenstein series $E_{k,\chi}$.  Our goal is to find a Siegel
modular form $G$ that is not a Saito-Kurokawa lift so that the Fourier coefficients of $G$ are congruent modulo $\varpi^{m}$ to those of $\Ff$ for some integer $m\geq 1$.  We are
able to produce such a $G$ by exploiting the explicit nature of the Saito-Kurokawa correspondence, using the pullback of a Siegel Eisenstein series and an inner product relation of Shimura (\cite{shimurazfctns}) as key ingredients.  It is at this step a significant improvement over the results of \cite{jimbrown} is made.  We use results of Garrett to show that in fact the Siegel Eisenstein series used pulls back to something cuspidal in each variable.  This allows us to avoid some of the ad hoc methods needed in \cite{jimbrown} that ``lost the powers of $p$''. 
 %We are also able to remove one of the technical restrictions on $p$ that a certain special value of a twisted $L$-function is a $p$-unit.  
Once we have the congruence desired, we generalize results of Urban (\cite{urbanduke}) to our situation to allow us to give the lower bound on the Selmer group.  This generalization is another significant improvement over our work in \cite{jimbrown} and allows us to get the lower bound desired.

%We conclude with a section of numerical examples where the conditions of the main theorem are shown to be satisfied and thus we obtain lower bounds for some concrete examples.

The author would like to thank Jim Cogdell for many helpful conversations.

\section{Notation}\label{sec:notation}

In this section we set the notation and definitions to be used throughout this paper. 

Let $\A$ be the ring of adeles over $\rat$.  For a prime $p$, we fix once and for all compatible embeddings $\ov{\rat} \hookrightarrow \ov{\rat}_{p}$, $\ov{\rat} \hookrightarrow \cmplx$, and $\ov{\rat}_{p} \hookrightarrow \cmplx$.  We write $\ord_{p}(n)$ to denote $p^{m}$ for $p^{m} \parallel n$.  We denote by $\varepsilon_{p}$ the $p$-adic cyclotomic character $\varepsilon_{p}: \absgal \rightarrow \GL_1(\inte_{p})$.  We drop the $p$ when it is clear from context.  We denote the residual representation of $\varepsilon_{p}$ by $\omega_{p}$, again dropping the $p$ when it is clear from context. 

For a ring $R$, we let $\M_{m,n}(R)$ denote the set of $m$ by $n$ matrices with entries in $R$.  If $m=n$ we write $\M_{m}(R)$ for $M_{m,m}(R)$. For a matrix $x \in \M_{2n}(R)$, we write
\begin{equation*}
x = \mat{a_{x}}{b_{x}}{c_{x}}{d_{x}}
\end{equation*}
where $a_{x}$, $b_{x}$, $c_{x}$, and $d_{x}$ are all in $\M_{n}(R)$.  We drop the subscript $x$ when it is clear from the context.  The transpose of a matrix $x$ is denoted by $^t\!x$.

Let $\SL_2$ and $\GL_{n}$ have their standard definitions.  We denote the complex upper half-plane by $\h{1}$.  We have the usual action of $\GL_2^{+}(\real)$ on $\h{1} \cup \mathbb{P}^1(\rat)$ given by linear fractional transformations.   Define 
\begin{equation*}
\Sp_{2n} = \{ g \in \GL_{2n}:\, ^{t}\!\gamma\iota_{2n} \gamma = \iota_{2n}\}, \quad \iota_{2n} = \mat{0_{n}}{-1_{n}}{1_{n}}{0_{n}}.
\end{equation*}
Siegel upper half-space is defined by
\begin{equation*}
\h{n} = \{ Z \in \M_{n}(\cmplx):\, ^{t}\!Z = Z, \IM(Z) > 0\}.
\end{equation*}
The group $\Sp_{2n}(\real)$ acts on $\h{n}$ via
\begin{equation*}
\mat{A}{B}{C}{D}Z = (AZ +B)(CZ+D)^{-1}.
\end{equation*}
We let $\Gamma_1^{J} = \SL_2(\inte) \ltimes \inte^2$ be the Jacobi modular group (\cite{eichlerzagier}).  

Given an $L$-function $L(s) = \prod_{p} L_{p}(s)$ and a finite set of places $\Sigma$, we write 
\begin{equation*}
L^{\Sigma}(s) = \prod_{p \notin \Sigma} L_{p}(s)
\end{equation*}
when we restrict to places away from $\Sigma$ and 
\begin{equation*}
L_{\Sigma}(s) = \prod_{p \in \Sigma} L_{p}(s)
\end{equation*}
when we restrict to the places in $\Sigma$. 

Let $\Gamma \subset \SL_2(\inte)$ be a congruence subgroup.  We write $M_{k}(\Gamma)$ to denote the space of modular forms of weight $k$ and level $\Gamma$.  We let $S_{k}(\Gamma)$ denote the subspace of cusp forms.  The $n$th Fourier coefficient of $f \in M_{k}(\Gamma)$ is denoted by $a_{f}(n)$.  Given a ring $R \subset \cmplx$, we write $M_{k}(\Gamma, R)$ for the space of modular forms with Fourier coefficients in $R$ and similarly for $S_{k}(\Gamma, R)$.  Let $f_1, f_2 \in M_{k}(\Gamma)$ with at least one of the $f_{i}$ a cusp form.  The Petersson product is given by
\begin{equation*}
\langle f_1, f_2 \rangle = \frac{1}{[\ov{\SL_2(\inte)}: \ov{\Gamma}]} \int_{\Gamma \backslash \h{1}} f_1(z) \ov{f_2(z)} y^{k-2} dx dy
\end{equation*}
where $\ov{\SL_2(\inte)} = \SL_2(\inte)/\{\pm 1_2\}$ and $\ov{\Gamma}$ is the image of $\Gamma$ in $\ov{\SL_2(\inte)}$.  The $n$th Hecke operator $T(n)$ has its usual meaning.  Let $A$ be a $\inte$-algebra.  Let $\T_{\inte}$ be the $\inte$-subalgebra of $\End_{\cmplx}(S_{k}(\SL_2(\inte))$ generated by $T(n)$ for $n=1,2, 3, \dots$.  Note that we do not include the weight in the notation as it will always be clear from context.  We set $\T_{A} = \T_{\inte} \otimes_{\inte} A$.  We say $f \in S_{k}(\SL_2(\inte))$ is a newform if it is an eigenform for all $T(n)$ and $a_{f}(1) = 1$.  The $L$-function associated to a newform $f$ of weight $k$ is given by
\begin{equation*}
L(s,f) = \sum_{n \geq 1} a_{f}(n) n^{-s}.
\end{equation*}
The $L$-function $L(s,f)$ can be factored as
\begin{equation*}
L(s,f) = \prod_{p}[(1-\alpha_{f}(p)p^{-s})(1-\beta_{f}(p)p^{-s})]^{-1}
\end{equation*}
where $\alpha_{f}(p) + \beta_{f}(p) = a_{f}(p)$ and $\alpha_{f}(p)\beta_{f}(p) = p^{k-1}$.  The terms $\alpha_{f}(p)$ and $\beta_{f}(p)$ are referred to as the $p$th Satake parameters of $f$.

Kohnen's $+$-space of half-integral weight modular forms is given by
\begin{equation*}
S_{k-1/2}^{+} (\Gamma_0(4)) = \{ g \in S_{k-1/2}(\Gamma_0(4)): \text{$a_{g}(n) = 0$ if $(-1)^{k-1}n \equiv 2,3 (\modu 4)$}\}.
\end{equation*}
The Petersson product on $S_{k-1/2}^{+}(\Gamma_0(4))$ is given by
\begin{equation*}
\langle g_1, g_2 \rangle = \int_{\Gamma_0(4) \backslash \h{1}} g_1(z) \ov{g_2(z)} y^{k-5/2} dx dy.
\end{equation*}

We denote the space of Jacobi cusp forms on $\Gamma_1^{\J}$ by $J_{k,1}^{\cusp}(\Gamma_1^{\J})$.  The inner product is given
by 
\begin{equation*}
\langle \phi_1, \phi_2 \rangle = \int_{\Gamma_1^{\J}\backslash \h{1} \times \cmplx} \phi_1(\tau, z)\overline{\phi_2(\tau,z)} v^{k-3} e^{-4 \pi y^2/v} dx \,dy\, du\, dv
\end{equation*}
for $\phi_1, \phi_2 \in J_{k,1}^{\cusp}(\Gamma_1^{\J})$ and $\tau= u + i v$, $z = x+iy$. 

We denote the space of Siegel modular forms of weight $k$ and level $\Gamma \subset \Sp_{2n}(\inte)$ by $\mathcal{M}_{k}(\Gamma)$.  The subspace of cusp forms is denoted by $\mathcal{S}_{k}(\Gamma)$.  For $F, G \in \mathcal{M}_{k}(\Sp_{2n}(\Gamma))$ with at least one a cusp form, the Petersson product is given by
\begin{equation*}
\langle F, G \rangle = \frac{1}{[ \ov{\Sp_{2n}(\inte)} : \ov{\Gamma}]} \int_{\Gamma \backslash \h{n}} F(Z) \ov{G(Z)} \det(\IM(Z))^{k} d\mu(Z).
\end{equation*}
We will be particularly interested in the decomposition  
$$\mathcal{S}_{k}(\Sp_{4}(\inte)) = \mathcal{S}_{k}^{\M}(\Sp_4(\inte)) \oplus \mathcal{S}_{k}^{\NM}(\Sp_4(\inte))$$
where $\mathcal{S}_{k}^{\M}(\Sp_4(\inte))$ is the space of Maass spezialchars and $\mathcal{S}_{k}^{\NM}(\Sp_4(\inte))$ is the orthogonal complement. A form $F \in \mathcal{S}_{k}(\Sp_4(\inte))$ is in $\mathcal{S}_{k}^{\M}(\Sp_4(\inte))$ if the Fourier coefficients of $F$ satisfy the relation
\begin{equation*}
A_{F}(n,r,m) = \sum_{d \mid \gcd(n,r,m)} d^{k-1} A_{F}\left(\frac{nm}{d^2}, \frac{r}{d}, 1 \right).
\end{equation*}

We let $T^{S}(n)$ denote the $n$th Siegel Hecke operator.  As above, we set $\T_{\inte}^{S}$ to be the $\inte$-subalgebra of $\End_{\cmplx}(\mathcal{S}_{k}(\Sp_{2n}(\inte))$ generated by the $T(n)$.  For a $\inte$-algebra $A$ we write $\T_{A}^{S} = \T_{\inte}^{S} \otimes_{\inte} A$. The Hecke algebra $\T_{\cmplx}^{S}$ respects the decomposition of $\mathcal{S}_{k}(\Sp_4(\inte))$ into the space of Maass and non-Maass forms (\cite{andrianov}).  

Let $F \in \mathcal{S}_{k}(\Gtwo)$ be a Hecke eigenform with eigenvalues $\lambda_{F}(m)$.  
Associated to $F$ is an $L$-function called the spinor $L$-function.  It is defined by 
\begin{equation*}
L_{\spin}(s,F) = \zeta(2s-2k+4) \sum_{m \geq 1} \lambda_{F}(m)m^{-s}.
\end{equation*}
One can also define the spinor $L$-function in terms of the Satake parameters $\alpha_0, \alpha_1$, and $\alpha_2$ of $F$.  One has 
\begin{equation*}
L_{\spin}(s,F) = \prod_{p}Q_{p}(p^{-s})^{-1}
\end{equation*}
where 
\begin{equation*}
Q_{p}(X) = (1-\alpha_0 X)(1-\alpha_0\alpha_1 X) (1-\alpha_0 \alpha_2 X) (1-\alpha_0 \alpha_2 \alpha_2 X).
\end{equation*}
The standard $L$-function
associated to $F$ is given by 
\begin{equation}\label{eqn:standardzeta}
L_{\st}(s,F) = \prod_{\ell} W_{\ell}(\ell^{-s})^{-1}
\end{equation}
where 
\begin{equation*}
W_{\ell}(t) = (1-\ell^{2}t) \prod_{i=1}^{2} (1- \ell^2 \alpha_{i}t)(1-\ell^2\alpha_{i}^{-1} t).
\end{equation*}
 Given a Hecke character $\phi$, the twisted standard zeta function is given by
\begin{equation*}
L_{\st}(s,F,\phi) = \prod_{\ell} W_{\ell}(\phi(\ell)\ell^{-s})^{-1}.
\end{equation*}

\section{Siegel Eisenstein series}\label{sec:eseries}

In this section we study pullbacks of Siegel Eisenstein series and show that the pullback of the Siegel Eisenstein series from $\Sp_{4n}$ to $\Sp_{2n} \times \Sp_{2n}$ is cuspidal in each variable.  We use the methods developed by 
Garrett (\cite{garrett}) along with standard facts about Siegel Eisenstein series to establish this result.  We also recall facts that can be found in \cite{jimbrown} about the Fourier coefficients of the Siegel Eisenstein series as well as a formula for the inner product of the pullback of the Siegel Eisenstein series with a cuspidal Siegel eigenform due to Shimura (\cite{shimurazfctns}).

\subsection{Basic definitions and results}

We begin by recalling the defintion of some
subgroups of $\Sp_{2m}(\A)$ and $\Sp_{2m}(\rat)$.  Let
$N>1$ be an integer let $\Sigma$ be the set of primes dividing $N$. For a prime $\ell$ define 
% \begin{equation*}
% D=D(N) = \Sp_{2n}(\real) \prod_{\ell \nmid \infty} K_{0,\ell}(N)
% \end{equation*}
% where
\begin{equation*}
K_{0,\ell}(N) = \left\{ g \in
\Sp_{2m}(\rat_{\ell}) : a_{g}, b_{g}, d_{g} \in \M_{m}(\inte_{\ell}), c_{g} \in \M_{m}(N\inte_{\ell}) \right\}
\end{equation*}
and set 
\begin{equation*}
K_{0,f}(N) = \prod_{\ell \nmid \infty} K_{0,\ell}(N).
\end{equation*}
Set
\begin{equation*}
K_{\infty} = \{ g \in \Sp_{2m}(\real) : g \i_{2m} = \i_{2m}\}
\end{equation*}
where $\i_{2m} = i 1_{2m}$. 
and 
\begin{equation*}
K_0(N) = K_{\infty} K_{0,f}(N).
\end{equation*} 
Set
\begin{equation*}
\mathbb{S}_{m} = \{ x \in \M_{m} : \, ^{t}x = x \}.
\end{equation*}
Let $P_{2m}=U_{2m}Q_{2m}$ be the Siegel parabolic of
$\Sp_{2m}$ defined by
\begin{equation*}
P_{2m} = \left\{ g \in \Sp_{2m} : c_{g} = 0 \right\}
\end{equation*}
with unipotent radical
\begin{equation*}
U_{2m} = \left\{ u(x) = \mat{1}{x}{0}{1} : x \in \mathbb{S}_{m} \right\}
\end{equation*}
and Levi subgroup
\begin{equation*}
Q_{2m} = \left\{ Q(A) = \mat{A}{0}{0}{^t\! A^{-1}} : A \in \GL_{m} \right\}.
\end{equation*}
We drop the subscript $2m$ from the notation when it is clear from context.

Let $k$ be a
positive integer such that $k > \max\{3, m+1\}$. Let $\chi$ be a Hecke character of 
$\mathbb{A}^{\times}$ satisfying
\begin{eqnarray}\label{eqn:character}
\chi_{\infty}(x) &=& \SGN(x)^{k},\\
\chi_{\ell}(a) &=& 1 \quad \text{if $\ell \nmid \infty$, $a \in
\inte_{\ell}^{\times}$, and $N \mid (a-1)$}. \nonumber
\end{eqnarray}

Define $\varepsilon(g,s;k,N,\chi)$ on $\Sp_{2m}(\A) \times \cmplx$ by $$\varepsilon(g,s;k,N,\chi) = 0$$ if $g \notin P(\A)K_{0}(N)$ and for $g = u(x)Q(A) \theta$ with $u(x) Q(A) \in P(\A)$ and $\theta \in K_{0}(N)$
\begin{equation*}
\varepsilon(g,s;k,N,\chi) = \varepsilon_{\infty}(g,s; k, \chi) \prod_{\ell \nmid N} \varepsilon_{\ell}(g,s;k,\chi) \prod_{\ell \mid N} \varepsilon_{\ell}(g,s; k, N, \chi)
\end{equation*}
where we define the components by
\begin{align*}
\varepsilon_{\infty}(g,s ; k, \chi) &= \chi_{\infty}(\det A_{\infty}) |\det A_{\infty}|^{2s} j^{k}(\theta_{\infty}, \i)^{-1},\\
\varepsilon_{\ell}(g,s; k, \chi) &= \chi_{\ell}(\det A_{\ell}) |\det A_{\ell}|^{2s} \quad \quad (\ell \nmid N) \\
\varepsilon_{\ell}(g,s;k,N,\chi) &= \chi_{\ell}(\det A_{\ell}) \chi_{\ell}(\det d_{\theta})^{-1} |\det A_{\ell}|^{2s} \quad \quad (\ell \mid N).
\end{align*}
The Siegel Eisenstein series is defined by 
\begin{equation*}
E(g,s) = E(g,s;k, N, \chi) = \sum_{\gamma \in P(\rat)\backslash
\Sp_{2m}(\rat)} \varepsilon(\gamma g, s; k,N,\chi).
\end{equation*}
The series $E(g,s)$ converges locally uniformly for $\Real(s) > (m+1)/2$ and can be continued to a meromorphic function
on all of $\cmplx$. It has a functional equation relating $E(g,s)$ to $E(g, (m+1)/2 - s)$.  

Associated to the Siegel Eisenstein series $E(g,s)$ is a complex version $E(Z,s)$ defined on $\h{m} \times \cmplx$ by 
\begin{equation*}
E(Z,s) = j^{k}(g_{\infty}, \i_{2m}) E(g,s)
\end{equation*}
where $Z = g_{\infty} \i_{2m}$ and $g = g_{\rat}g_{\infty}\theta_{f} \in \Sp_{2m}(\rat)\Sp_{2m}(\real) K_{0,f}(N)$.  It will be important for us that $E(Z, (m+1)/2 - k/2)$ and $E(Z, k/2)$ are both holomorphic modular forms of weight $k$ and level $N$ (\cite{shimmathann}).  We will use whichever of 
these Siegel Eisenstein series is most convenient for the current application, keeping in mind the above relation between them.

Finally, we recall a result on the Fourier coefficients of $E(Z,(m+1)/2 -
k/2)$ demonstrated in \cite{jimbrown}.  Let $E^{*}(g,s) = E(g \varsigma_{f}^{-1}, s)$ where we recall $\varsigma = \mat{0_{m}}{-1_{m}}{1_{m}}{0_{m}}$.  It is convenient to look at this translation when calculating the Fourier coefficients of $E(g,s)$.  We let $E^{*}(Z,s)$ be the corresponding complex version.  Write elements $Z \in \h{m}$ as $Z = X+iY$ with $X,Y \in \mathbb{S}^{m}(\real)$ and $Y> 0$.  $L=\mathbb{S}^{m}(\rat)\cap \M_{m}(\inte)$, $L' = \{\mathfrak{s} \in
 \mathbb{S}^{m}(\rat):\Tr(\mathfrak{s}L)\subseteq \inte\}$ and $M = N^{-1}L'$.
The Eisenstein series $E^{*}(Z,s)$ has a Fourier expansion
 \begin{equation*}
 E^{*}(Z,s) = \sum_{h\in M} a(h,Y,s)
 e(\Tr(h X))
 \end{equation*}
 for $Z = X+ i Y\in \h{m}$.  Set $$\Lambda^{\Sigma}(s,\chi) = L^{\Sigma}(2s,\chi) \prod_{j=1}^{[m/2]}L^{\Sigma}(4s-2j,\psi^2).$$
 We normalize $E^{*}(Z,s)$ by multiplying it by
$\pi^{-\frac{m(m+2)}{4}} \Lambda^{\Sigma}(s,\chi)$.  We have the following result.

\begin{theorem} (\cite{jimbrown} Theorem 2.4) Let $p$ be an odd prime such that $p > m$ and $\gcd(p,N) =1$.  Set 
\begin{equation}\label{eqn:normalizedeseries}
D_{E^{*}}(Z,(m+1)/2-k/2) = (\pi^{-\frac{m(m+2)}{4}} \Lambda^{\Sigma}((m+1)/2-k/2,\chi)) E^{*}(Z, (m+1)/2 - k/2).
\end{equation}
Then 
\begin{equation*}
D_{E^{*}}(Z,(m+1)/2-k/2) \in \mathcal{M}_{k}(\Gamma_0^{2m}(N), \intep[\chi, i^{mk}]).
\end{equation*}
\end{theorem}

\subsection{Pullbacks of Siegel Eisenstein series}

We begin by reviewing the notion of the pullback of an automorphic form on $\Sp_{4n}(\A)$ to $\Sp_{2n}(\A) \times \Sp_{2n}(\A)$.  This theory has been well established and the interested reader is advised to consult any of the following references for more details: (\cite{boch}, \cite{garrett}, \cite{garrett2}, \cite{shimurazfctns}, \cite{shimurabook}).

Let $\Sp_{2n} \times \Sp_{2n}$ be imbedded in $\Sp_{4n}$ via
\begin{equation*}
\iota\left( \mat{a_1}{b_1}{c_1}{d_1} , \mat{a_2}{b_2}{c_2}{d_2} \right) = \begin{pmatrix} a_1 & 0 & b_1 & 0 \\
        0 & a_2 & 0 & b_2 \\ c_1 & 0 & d_1 & 0 \\ 0 & c_2 & 0 & d_2 \end{pmatrix}.
\end{equation*}
One can show that given a holomorphic automorphic form $F$ of weight $k$ and level $N$ on $\Sp_{4n}(\A)$, the function $(g_1, g_2) \mapsto F(\iota(g_1, g_2))$ is a holomorphic automorphic form of weight $k$ and level $N$ in each variable, see \cite{garrett} for example.  The form $F(\iota(g_1, g_2))$ is referred to as the pullback of $F$ from $\Sp_{4n}$ to $\Sp_{2n} \times \Sp_{2n}$, or just the pullback of $F$.  We will often drop the $\iota$ from the notation when it is clear from context.  If we wish to work classically rather then adelically we make use of the embedding \begin{equation*}
\h{n} \times \h{n} \hookrightarrow \h{2n}
\end{equation*}
given by 
\begin{equation*}
Z \times W \mapsto \mat{Z}{0}{0}{W} = \diag[Z,W]
\end{equation*}
arising from the isomorphism $\Sp_{2n}(\real)/K_{0,\infty} \cong \h{n}$. 

Our goal is to show that the Eisenstein series $E(\iota(g_1, g_2), (n+1)/2 - k/2)$ is cuspidal in $g_1$ and $g_2$ where $g_1, g_2 \in \Sp_{2n}(\A)$. We begin by noting that using the functional equation, this is equivalent to showing that $E(\iota(g_1,g_2), k/2)$ is cuspidal in $g_1$ and $g_2$.  We now wish to apply the methods used by Garrett in \cite{garrett}.  To do this, we restrict our Eisenstein series from $P_{4n}(\A)K_{0}(N)$ to $P_{4n}(\A)K(N)$ where $K(N) = K_{\infty} \prod_{\ell \nmid \infty} K_{\ell}(N)$ with 
\begin{equation*}
K_{\ell}(N) = \{ g \in \Sp_{4n}(\rat_{\ell}) : g \equiv 1_{4n} (\modu N)\}.
\end{equation*}
Classically we are taking $E(\mathcal{Z}, k/2)$, which is in $\mathcal{M}_{k}(\Gamma_0^{4n}(N))$, and thinking of it as a modular form in $\mathcal{M}_{k}(\Gamma^{4n}(N))$.  From this it is clear that if the pullback of the restricted Eisenstein series is cuspidal, so is the original.  We denote the Eisenstein series restricted to $P_{4n}(\A)K(N)$ as $E'$.  Observe that if we write 
\begin{equation*}
E'(g,k/2) = \sum_{\gamma \in P_{4n}(\rat)\backslash \Sp_{4n}(\rat)} \epsilon(\gamma g, k/2)
\end{equation*}
with $\epsilon(g,k/2)$ factored into local components as in the previous section, then we have $\epsilon_{\infty}(g,k/2) = \varepsilon_{\infty}(g,k/2)$, $\epsilon_{\ell}(g,k/2) = \varepsilon_{\ell}(g,k/2)$ for $\ell \nmid N$, and $\epsilon_{\ell}(g,k/2) = \chi_{\ell}(\det A_{\ell}) |\det A_{\ell}|^{k}$ for $\ell \mid N$.  Note the $\chi_{\ell}(\det d_{\theta})^{-1}$ drops out because of the restriction to $K(N)$.  We are now in a position to apply Garrett's argument to $E'(\iota(g_1, g_2), k/2)$ to show that it is cuspidal in $g_1$ and $g_2$ (\cite{garrett}).  

Let $\theta \in \Sp_{4n}(\A_{f})$ be an element so that 
\begin{equation*}
(0_{2n} 1_{2n}) \theta = \begin{pmatrix} 1_{n} & 1_{n} & 0_{n} & 0_{n} \\ 0_{n} & 0_{n} & -1_{n} & 1_{n} \end{pmatrix}.
\end{equation*}
We will consider the Eisenstein series translated by this $\theta$, namely, we will show that $E'(\iota(g_1, g_2) \theta^{-1}, k/2)$ is cuspidal in $g_1$ and $g_2$.  Classically this just amounts to translating to a different cusp, so it is sufficient to show this translated Eisenstein series is cuspidal in $g_1$ and $g_2$. 

The space 
\begin{equation*}
X = P_{4n}(\rat) \backslash \Sp_{4n}(\rat) / \iota(\Sp_{2n}(\rat) \times \Sp_{2n}(\rat))
\end{equation*}
has representatives $\tau_0, \tau_1, \dots, \tau_{n}$ so that 
\begin{equation*}
(0_{2n} 1_{2n}) \tau_{i} = \begin{pmatrix} 1_{n-i} & 0_{i} & 0_{n-i} & 0_{i} & 0_{n-i} & 0_{i} & 0_{n-i} & 0_{i} \\
            0_{n-i} & 1_{i} & 0_{n-i} & 1_{i} & 0_{n-i} & 0_{i} & 0_{n-i} & 0_{i} \\
            0_{n-i} & 0_{i} & 0_{n-i} & 0_{i} & 0_{n-i} & 0_{i} & 1_{n-i} & 0_{i} \\
            0_{n-i} & 0_{i} & 0_{n-i} & 0_{i} & 0_{n-i} & -1_{i} & 0_{n-i} & 1_{i} \end{pmatrix}.
\end{equation*}
This is essentially a result in geometric algebra, see \cite{garrett} for a proof.  Let $I_{i}$ be the isotropy group of $P_{4n}(\rat)\tau_{i}$ in $P_{4n}(\rat) \backslash \Sp_{4n}(\rat)$ under $\Sp_{2n}(\rat) \times \Sp_{2n}(\rat)$ acting on the right, i.e., 
\begin{equation*}
I_{i} = \{(g_1, g_2) \in \Sp_{2n}(\rat) \times \Sp_{2n}(\rat) : P_{4n}(\rat) \tau_{i} \iota(g_1, g_2)  = P_{4n}(\rat) \tau_{i} \}.
\end{equation*}
This allows us to write our Eisenstein series as 
\begin{equation*}
E'(\iota(g_1, g_2) \theta^{-1}, k/2) = \sum_{0 \leq i \leq n} \vartheta_{i}(g_1, g_2)
\end{equation*}
where 
\begin{equation*}
\vartheta_{i}(g_1, g_2) = \sum_{\gamma \in I_{i} \backslash \iota(\Sp_{2n}(\rat) \times \Sp_{2n}(\rat))} \epsilon(\tau_{i} \gamma \iota(g_1, g_2) \theta^{-1}, k/2).
\end{equation*}
Our goal now is to show that $\vartheta_{i}(g_1,g_2) = 0$ for all $g_1, g_2 \in \Sp_{2n}(\rat)$ unless $i=n$.  Observe that it is enough to show that for any prime $p \mid N$, we have $\epsilon_{p}(\tau_{i} \gamma \iota(g_1, g_2) \theta^{-1}, k/2)= 0$ for all $g_1, g_2 \in \Sp_{2n}(\rat_{p})$   In order to show this, we give an integral representation of the Eisenstein series components $\epsilon_{p}$.  

Let $\GL_{n}(\rat_{p})$ have Haar measure normalized so that $\GL_{n}(\intep)$ has Haar measure 1.  For $p \mid N$, let $\phi_{p}$ be the characteristic function of the set 
\begin{equation*}
\{ (m_1 m_2) \in \M_{n \times 2n}(\rat_{p}) : (m_1 m_2) \equiv (0_{n} 1_{n}) (\modu N) \}.
\end{equation*}
%If $p$ is a finite prime that does not divide $N$, let $\phi_{p}$ be the characteristic function of $\M_{n, 2n}(\intep)$. 
A change of variables shows that for $g \in \Sp_{n}(\rat_{p})$ we have 
\begin{equation*}
\epsilon_{p}(g; k/2) = \mathcal{I}_{p}(g)/\mathcal{I}_{p}(1_{2n})
\end{equation*}
where  
\begin{equation*}
\mathcal{I}_{p}(g) = \int_{\GL_{n}(\rat_{p})} |\det h|^{k} \chi(\det h) \phi_{p}(h (0_{n} 1_{n}) g) dh.
\end{equation*}

In order to show that $\vartheta_{i}(g_1, g_2) = 0$ for all $g_1, g_2 \in \Sp_{2n}(\rat)$ unless $i = n$, we will show that 
\begin{equation*}
\mathcal{I}_{p}(\tau_{i} \iota(g_1, g_2) \theta^{-1}) = 0 
\end{equation*} 
for $p \mid N$ and every $g_1, g_2 \in \Sp_{2n}(\rat_{p})$ if $i \neq n$.  In particular, the definition of $\mathcal{I}_{p}$ allows us to reduce this to showing that 
\begin{equation*}
\phi_{p}(g(0_{2n}1_{2n}) \tau_{i} \iota(g_1,g_2) \theta^{-1}) = 0
\end{equation*}
for every $p \mid N$, $g_1, g_2 \in \Sp_{2n}(\rat_{p})$, and $g \in \GL_{2n}(\rat_{p})$ unless $i = n$.  We have that $\phi_{p}(g(0_{2n}1_{2n}) \tau_{i} \iota(g_1,g_2) \theta^{-1}) \neq 0$ if and only if $g(0_{2n}1_{2n}) \tau_{i} \iota(g_1,g_2) \theta^{-1} \equiv (0_{2n} 1_{2n}) (\modu N)$, i.e., 
\begin{equation*}
g(0_{2n}1_{2n}) \tau_{i} \iota(g_1,g_2) \equiv (0_{2n} 1_{2n}) \theta (\modu N).
\end{equation*}

Define $\psi: M_{2n,4n} \rightarrow M_{2n}$ by 
\begin{equation*}
\begin{pmatrix} a_{11} & a_{12}& b_{11} & b_{12} \\ c_{11} & c_{12} & d_{11} & d_{12} \end{pmatrix} \mapsto \mat{a_{11}}{b_{11}}{c_{11}}{d_{11}}.
\end{equation*}
The definition of $\theta$ yields 
\begin{equation*}
\psi((0_{2n} 1_{2n})\theta) = \mat{1_{n}}{0_{n}}{0_{n}}{-1_{n}}.
\end{equation*}
Thus, we have that $\psi((0_{2n} 1_{2n})\theta)$ is a matrix of rank $2n$ in $M_{2n}(\rat_{p})$.  However, a short calculation yields that 
\begin{equation*}
\psi(g(0_{2n}1_{2n})\tau_{i} \iota(g_1,g_2)) = g \psi((0_{2n} 1_{2n}) \tau_{i}) g_1
\end{equation*}
and 
\begin{equation*}
\psi((0_{2n} 1_{2n}) \tau_{i}) = \begin{pmatrix} 1_{n-i} & 0_{i} & 0_{i} & 0_{i} \\ 0_{n-i} & 1_{i} & 0_{i} & 0_{i} \\
        0_{n-i} & 0_{i} & 0_{i} & 0_{i} \\ 0_{n-i} & 0_{i} & 0_{i} & -1_{i} \end{pmatrix}.
\end{equation*}
Thus, the rank of $\psi(g(0_{2n}1_{2n})\tau_{i} \iota(g_1,g_2))$ is at most $n+i$.  Hence, for $\psi(g(0_{2n}1_{2n})\tau_{i} \iota(g_1,g_2))$ to be congruent modulo $N$ to $(0_{2n} 1_{2n})\theta$ we must have a matrix of rank at most $n+i$ congruent modulo a proper ideal to a matrix of rank $2n$.  This can only happen if $i = n$.  Thus, we have shown that 
\begin{equation*}
\psi_{p}(g(0_{2n}1_{2n})\tau_{i} \iota(g_1, g_2) \theta^{-1}) = 0
\end{equation*}
unless $i = n$.  This shows that $\vartheta_{i} (g_1, g_2) = 0$ for all $g_1, g_2 \in \Sp_{2n}(\rat)$ unless $i = n$.  Thus we have 
\begin{equation*}
E'(\iota(g_1, g_2) \theta^{-1}, k/2) = \sum_{ \gamma \in I_{n} \backslash \iota(\Sp_{2n}(\rat) \times \Sp_{2n}(\rat))} \epsilon(\tau_{n} \gamma \iota(g_1,g_2) \theta^{-1}, k/2).
\end{equation*}

We are now in a position to show that $E'(\iota(g_1, g_2)\theta^{-1}, k/2)$ is a cuspform in $g_1$ ($g_2$ respectively).  We show here that $E'(\iota(g_1,g_2)\theta^{-1}, k/2)$ is a cusp form in $g_1$.  The argument to establish the result for $g_2$ is the same argument and is omitted. Observe that we have
\begin{equation}\label{eqn:isotropygroup}
I_{n} = \{\iota(g_1, g_2) : g_2 = \hat{g}_1\} \cong \Sp_{2n}(\rat)
\end{equation}
where $\hat{g} = \beta g \beta$ with $\beta = \mat{-1_{n}}{0_{n}}{0_{n}}{1_{n}}$.  Thus, we can take $\iota(\Sp_{2n}(\rat) \times \{1\})$ as a collection of representatives for $I_{n} \backslash \iota(\Sp_{2n}(\rat) \times \Sp_{2n}(\rat))$.  This allows us to write 
\begin{equation*}
E'(\iota(g_1, g_2) \theta^{-1}, k/2) = \sum_{\gamma \in \Sp_{2n}(\rat)} \epsilon(\tau_{n} \iota(\gamma g_{1}, g_2) \theta^{-1}, k/2).
\end{equation*}
We can unfold this sum along the unipotent radical as follows:
\begin{align*}
E'(\iota(g_1, g_2) \theta^{-1}, k/2) &= \sum_{\gamma \in \Sp_{2n}(\rat)} \epsilon(\tau_{n} \iota(\gamma g_{1}, g_2) \theta^{-1}, k/2)\\
    &= \sum_{\gamma \in \Sp_{2n}(\rat)/U_{2n}(\rat)} \sum_{u \in U_{2n}(\rat)} \epsilon(\tau_{n} \iota(\gamma u g_1, g_2)\theta^{-1}, k/2)\\
    &= \sum_{\gamma \in \Sp_{2n}(\rat)/U_{2n}(\rat)} \sum_{u \in U_{2n}(\rat)} \epsilon(\tau_{n} \iota(ug_1, \hat{\gamma}^{-1} g_2) \theta^{-1}, k/2)
\end{align*}
where the last equality follows from the isomorphism in equation (\ref{eqn:isotropygroup}).  We can expand the sum 
\begin{equation*}
\sum_{u \in U_{2n}(\rat)} \epsilon(\tau \iota(ug_1, \hat{\gamma}^{-1} g_2) \theta^{-1}, k/2)
\end{equation*}
in its Fourier expansion in $g_1$ along the unipotent radical.  Since we are only interested here in showing cuspidality, it is enough to show that 
the $\sigma$-th Fourier coefficient is zero for all $g_1, g_2 \in \Sp_{2n}(\A)$ unless $\sigma$ is totally positive definite.  Recall the Fourier coefficient is given by 
\begin{equation*}
\int_{U_{2n}(\rat) \backslash U_{2n}(\A)} \overline{\psi}(\Tr(\sigma g)) \sum_{u \in U_{2n}(\rat)} \epsilon(\tau_{n} \iota(u u(g) g_1, g_2) \theta^{-1}, k/2) dg
\end{equation*}
where we recall that $\psi$ is the standard additive character on $\A/\rat$ given by $x \mapsto e^{2\pi i x}$ on $\real$ and is trivial on $\intep$ for all $p$.   We can fold this integral to obtain the collapsed integral given by
\begin{equation*}
\int_{U_{2n}(\A)} \bar{\psi}(\Tr(\sigma g)) \epsilon(\tau_{n} \iota(u(g) g_1, g_2) \theta^{-1}, k/2) dg.
\end{equation*}
We restrict ourselves now to looking at the infinite place as this is the place that forces the vanishing of the integral for $\sigma$ not totally positive definite.  Recall that $\theta \in \Sp_{4n}(\A_{f})$, so when we look at the infinite place our integral becomes
\begin{equation}\label{eqn:realintegral}
\int_{U_{2n}(\real)} \bar{\psi}(\Tr(\sigma g)) \epsilon_{\infty}(\tau_{n} \iota(u(g) g_1, g_2), k/2) dg.
\end{equation}
It is enough to consider $g_1$ and $g_2$ of the form $g_1 = Q(A_1)$, $g_2 = u(x)Q(A_2)$.  This follows from the Iwasawa decomposition and the right $(K(N), j^{-k})$-equivariance of $\epsilon_{\infty}$.  One can use the fact that $\tau_{n}$ is of the form 
\begin{equation*}
\begin{pmatrix} * & * & * & * \\ * & * & * & * \\ 1_{n} & 1_{n} & 0_{n} & 0_{n} \\ 0_{n} & 0_{n} & -1_{n} & 1_{n} \end{pmatrix}
\end{equation*} 
to calculate 
\begin{equation*}
\tau_{n} \iota(u(g) g_1, g_2)  = \begin{pmatrix} * & * & * & * \\ * & * & * & * \\ A_1 & A_2 & g\, ^t\! A_1^{-1} & x\, ^t\! A_2^{-1} \\ 0 & 0 & -^t\! A_1^{-1} & ^t\! A_2^{-1} \end{pmatrix}
\end{equation*}
where we use a $*$ to indicate that we are not interested in this entry of the matrix.  We now make the observation that for any $g \in \Sp_{4n}(\real)$, one has $\epsilon_{\infty}(g, k/2) = \chi(\det A_{g}) j(g, i)^{-k}$ where $g = u(x_{g})Q(A_{g})k_{g}$.  This follows immediately from the definition of $\epsilon_{\infty}$ and the fact that $j(g,i)^{-k} = |\det A_{g}|^{k} j(k_{g},i)^{-k}$.  Applying this to our current situation we have
\begin{equation*}
\epsilon_{\infty}(\tau_{n}\iota(u(g)g_1, g_2), k/2) = \chi(\det \tilde{A}) \det \left( i \mat{A_1}{A_2}{0}{0} + \mat{g\, ^t\! A_1^{-1}} {x\, ^t\! A_2^{-1}}{-^t\! A_1^{-1}}{^t\! A_2^{-1}} \right)^{-k}
\end{equation*}
where we write $\tilde{A} = A_{\tau_{n}\iota(u(g)g_1, g_2)}$ to ease the notation.  Observe that we have
\begin{equation*}
i \mat{A_1}{A_2}{0}{0} + \mat{g\, ^t\! A_1^{-1}} {x\, ^t\! A_2^{-1}}{-^t\! A_1^{-1}}{^t\! A_2^{-1}} = \mat{iA_1 \, ^t\!A_1 + g}{i A_2 \, ^t\! A_2 + x}{-1}{1} \mat{^t\!A_1^{-1}}{0}{0}{ ^t\!A_2^{-1}}.
\end{equation*}
We now recall the formula that if $D^{-1}$ exists then we have
\begin{equation*}
\det \mat{A}{B}{C}{D} = \det D \det( A- B D^{-1}C).
\end{equation*}
This follows immediately from the identity
\begin{equation*}
\mat{A}{B}{C}{D} = \mat{1}{BD^{-1}}{0}{1} \mat{A-BD^{-1}C}{0}{C}{D}.
\end{equation*} 
Thus, we have 
\begin{align*}
\epsilon_{\infty}(\tau_{n}\iota(u(g)g_1, g_2), k/2) &= \chi(\det \tilde{A}) (\det A_1 \det A_2)^{-1} \det(g + x + i A_1 \, ^t\!A_1 + i A_2 \, ^t\!A_2)^{-k}\\
     &= \chi(\det \tilde{A}) (\det A_1 \det A_2)^{-1} \det (g + z)^{-k}
\end{align*}
where $z \in \h{n}$.  Thus, the integral in equation (\ref{eqn:realintegral}) is equal to
\begin{equation*}
\chi(\det \tilde{A}) (\det A_1 \det A_2)^{-1} \int_{U_{2n}(\real)} \overline{\psi}(\Tr(\sigma g)) \det(g + z)^{-k} dg \\
    = 0
\end{equation*}
unless $\sigma$ is totally positive definite.  This last equality is a result due to Siegel.  Thus, we have that $E'(\iota(g_1, g_2) \theta^{-1}, k/2)$ is a holomorphic cuspform in $g_1$.  Similarly, we have it is a holomorphic cuspform in $g_2$ as well.  Combining this with our previous discussion relating $E(\iota(g_1, g_2), (n+1)/2 - k/2)$ to $E'(\iota(g_1, g_2)\theta^{-1}, k/2)$, we see that $E(\iota(g_1, g_2), (n+1)/2 - k/2)$ is a holomorphic cuspform in $g_1$ (respectively, $g_2$.)

\subsection{Pullbacks and an inner product relation}

In this section we summarize a result of Shimura giving the inner product between a cuspidal Siegel eigenform and 
the pullback of the Siegel Eisenstein series.  See \cite{shimurazfctns} for a detailed treatment of this material.  We specialize to the case
of $\Sp_4 \times \Sp_4$ embedded in $\Sp_8$ as this will be the case for which we apply the result.  

Let $\sigma \in \Sp_{8}(\A_{f})$ be defined by
\begin{equation*}
\sigma_{\ell} = \left \{ \begin{array}{ccc} I_{8} & & \text{if $\ell \nmid N$} \\
\begin{pmatrix}I_{4} & 0_{4} \\ \begin{pmatrix} 0_{2} & I_{2} \\ I_{2} &
0_{2} \end{pmatrix} & I_{4} \end{pmatrix} & & \text{if $\ell \mid N$.}
\end{array}\right.
\end{equation*}
Strong approximation gives an element $\rho \in
\Sp_{8}(\rat)$ so that
 $E\Slash{\rho}(Z,s)$
corresponds to $E(g\sigma^{-1},s)$.

Let $F \in \mathcal{S}_{k}(\Gamma_0^{4}(N))$ be a Siegel
eigenform.  Applying (\cite{shimurazfctns}, equation (6.17)) to our
situation we obtain
\begin{equation}\label{eqn:shimuraeq1}
\langle D_{E\Slash{\rho}}(\diag[Z,W], (5-k)/2),
(F\Slash{\varsigma})^{c}(W) \rangle = \pi^{-3} \,\mathcal{A}_{k,N}
L^{\Sigma}_{\st}(5-k,F, \chi) F(Z)
\end{equation}
where $D_{E\Slash{\rho}}$ is the normalized Eisenstein series as defined in equation (\ref{eqn:normalizedeseries}), $\displaystyle \mathcal{A}_{k,N} =  \frac{(-1)^{k}\,
2^{2k-3} v_{N}}{3 \,[\Sp_{4}(\inte):\,\Gamma_0^{4}(N)]}$,
$v_{N} = \pm 1$,
$L^{\Sigma}_{\st}(5-k,F,\chi)$ is the standard zeta function as
defined in equation (\ref{eqn:standardzeta}), and
$(F\Slash{\varsigma})^{c}$ denotes taking the complex conjugates of
the Fourier coefficients of $F\Slash{\varsigma}.$ 

\section{The Saito-Kurokawa Correspondence}\label{sec:sklifts}

The Saito-Kurokawa correspondence is a well known result with several excellent references
for the basic facts (\cite{eichlerzagier}, \cite{geer}, \cite{zagier}).  For the facts we will be interested in, the reader
is advised to consult the section on Saito-Kurokawa lifts in \cite{jimbrown}.  Here we will briefly recall
the Saito-Kurokawa correspondence and state some necessary facts, but omit a lengthy discussion preferring to point 
unfamiliar readers to the appropriate references.

Let $f \in S_{2k-2}(\SL_2(\inte), \ROI)$ be a newform with Fourier coefficients in a ring $\ROI$.  The Saito-Kurokawa correspondence
associates to $f$ a cuspidal Siegel eigenform $\Ff$ in the Maass spezialchars $\mathcal{S}_{k}^{\M}(\Sp_4(\inte))$. This correspondence is achieved via a set of isomorphisms, the first being between $S_{2k-2}(\SL_2(\inte))$ and $S_{k-1/2}^{+}(\Gamma_0(4))$.  The isomorphism from Kohnen's $+$-space of half-integral weight modular forms to $S_{2k-2}(\SL_2(\inte))$ is given by
sending 
\begin{equation*}
g(z) = \sum_{\substack{n \geq 1 \\ (-1)^{k-1}n \equiv 0,1(\modu 4)}} c_{g}(n) q^{n} \in S_{k-1/2}^{+}(\Gamma_0(4))
\end{equation*}
to 
\begin{equation*}
\zeta_{D} g(z) = \sum_{n=1}^{\infty} \left(\sum_{d \mid n} \left(\frac{D}{d}\right) d^{k-2} c_{g}(|D|n^2/d^2)\right) q^{n}
\end{equation*}
where $D$ is a fundamental discriminant with $(-1)^{k-1}D>0$.  The second isomorphism is between $J_{k,1}^{\cusp}(\SL_2(\inte) \ltimes \inte^2)$ and $S^{+}_{k-1/2}(\Gamma_0(4))$ via the map 
\begin{equation*}
\sum_{\substack{D<0, r \in \inte \\D \equiv r^2 (\modu 4)}} c(D,r) e\left(\frac{r^2-D}{4}\tau + rz\right) \mapsto \sum_{\substack{D<0 \\ D \equiv 0,1 (\modu 4)}} c(D)e(|D|\tau).
\end{equation*}
Finally, we obtain an isomorphism between $J_{k,1}^{\cusp}(\SL_2(\inte) \ltimes \inte^2)$ and the space of Maass spezialchars $\mathcal{S}_{k}^{\M}(\Sp_4(\inte))$ given by 
\begin{equation*}
\phi(\tau,z) \mapsto F(\tau, z, \tau') = \sum_{m \geq 0} V_{m}\phi(\tau, z)e(m\tau')
\end{equation*}
where $V_{m}$ is the index shifting operator as defined in (\cite{eichlerzagier}, Section 4).  We will need that the Fourier coefficients of the form $F$ arising from $\phi(\tau,z)$ are related to the Fourier coefficients of $\phi$ by 
\begin{equation*}
A_{F}(n,r,m) = \sum_{d \mid \gcd(m,m,r)} d^{k-1} c_{\phi}\left(\frac{4nm-r^2}{d^2}, \frac{r}{d}\right).
\end{equation*}

We recall that we write $\mathcal{S}_{k}^{\NM}(\Sp_4(\inte))$ for the orthogonal complement to $\mathcal{S}_{k}^{\M}(\Sp_4(\inte))$ in $\mathcal{S}_{k}(\Sp_4(\inte))$ and refer to this as the space of non-Maass forms. The Saito-Kurokawa correspondence is stated in the following theorem.

\begin{theorem}\label{thm:skcorrespondence} (\cite{zagier}) There is a Hecke-equivariant isomorphism between $S_{2k-2}(\SL_2(\inte))$ and $\mathcal{S}_{k}^{\M}(\Sp_4(\inte))$ such that if $f \in S_{2k-2}(\SL_2(\inte))$ is a newform, then one has
\begin{equation}\label{eqn:skequation}
L_{\spin}(s,\Ff) = \zeta(s-k+1)\zeta(s-k+2)L(s,f).
\end{equation}
\end{theorem}

As our applications of the Saito-Kurokawa correspondence will be in terms of arithmetic data, we also need the follow two results.

\begin{corollary}\label{thm:skfourier}(\cite{jimbrown}, Corollary 3.8) Given $f \in S_{k}(\SL_2(\inte), \ROI)$ a newform, then $\Ff$ also has Fourier coefficients in $\ROI$.  In particular, if $\ROI$ is a discrete valuation ring, 
$\Ff$ has a Fourier coefficient in $\ROI^{\times}$.
\end{corollary}

\begin{corollary}\label{thm:skcomplexconjugation} Let $f \in S_{k}(\SL_2(\inte))$ and $\Ff$ the Saito-Kurokawa lift of $f$.  One has that
$\Ff^{c} = F_{f^{c}}$.  In other words, the Saito-Kurokawa correspondence respects complex conjugation of Fourier coefficients.
\end{corollary}

\begin{proof} This is straightforward by viewing the lift through the isomorphisms described above.  For example, if we let $g_{f}$ denote the half-integral weight modular form corresponding to $f$, then just observe that $g_{f}^{c}$ maps to $f^{c}$ under the given isomorphism, as does $g_{f^{c}}$.  Thus, we must have $g_{f^{c}} = g_{f}^{c}$.  The others are even more obvious. 
\end{proof} 

We will also make use of the following three results. 

\begin{theorem}\label{thm:innerproducts} (\cite{kohnenskoruppa}, \cite{kohnenzagier}) Let $f \in S_{2k-2}(\SL_2(\inte))$ be a newform, $\Ff \in \mathcal{S}_{k}^{\M}(\Sp_4(\inte))$
the corresponding Saito-Kurokawa lift, and $g(z) = \sum c_{g}(n)q^{n}$ the weight $k-1/2$ cusp form corresponding
to $f$.  We have the following inner product relation
\begin{equation*}
\langle \Ff, \Ff \rangle = \frac{(k-1)}{2^5 3^2 \pi }\, \cdot \frac{ c_{g}(|D|)^2}{ |D|^{k-3/2}}\cdot \, \frac{L(k,f)}{L(k-1,f,\chi_{D})} \, \langle f, f \rangle
\end{equation*}
where $D$ is a fundamental discriminant so that $(-1)^{k-1} D > 0$ and $\chi_{D}$ is the character associated to the field $\rat(\sqrt{D})$.
\end{theorem}

\begin{theorem}\label{thm:standardzetafactorization}(\cite{jimbrown}, Theorem 3.10) Let $N$ be a positive integer, $\Sigma$ the set of primes dividing $N$, and $\chi$ a 
Dirichlet character of conductor $N$.  
Let $f \in S_{2k-2}(\SL_2(\inte))$ be a newform and $\Ff$ the corresponding Saito-Kurokawa lift of $f$.  The standard zeta function of $\Ff$ factors 
as 
\begin{equation*}
L^{\Sigma}_{\st}(2s, \Ff, \chi) = L^{\Sigma}(2s-2,\chi) L^{\Sigma}(2s+k-3,f,\chi)L^{\Sigma}(2s+k-4,f,\chi).
\end{equation*}
\end{theorem}

\begin{proposition} Let $f \in S_{2k-2}(\SL_2(\inte))$ be a newform, $\rho_{f}$ the associated $\ell$-adic Galois representation, $\Ff$ the Saito-Kurokawa lift, and $\rho_{\Ff}$ the associated 4-dimensional $\ell$-adic Galois representation.  Then one has
\begin{equation*}
\rho_{\Ff} = \begin{pmatrix} \varepsilon^{k-2} & & \\ & \rho_{f} & \\ & & \varepsilon^{k-1} \end{pmatrix}
\end{equation*}
where $\varepsilon$ is the $\ell$-adic cyclotomic character and blank spaces in the matrix are assumed to be $0$'s of the appropriate size.
\end{proposition}

\begin{proof} This fact is used in \cite{jimbrown}, though not formally stated as a proposition there.  It follows directly from the decomposition of the Spinor $L$-function of $\Ff$, the fact that $\rho_{\Ff}$ is necessarily semi-simple, and the Brauer-Nesbitt theorem. 
\end{proof}

%  The Maass spezialchars are defined as follows.  Let $F \in \mathcal{S}_{k}(\Sp_4(\inte))$. Since we are working with $\Sp_4(\inte)$, the Fourier expansion of $F$ can be written as 
% \begin{equation*}
% F(\tau, z, \tau') = \sum_{\substack{m,n,r \in \inte \\ m,n,4mn-r^2 \geq 0}} A_{F}(n,r,m)e(n\tau + rz + m \tau').
% \end{equation*} 
% To relate this back to the Fourier expansion encountered thus far, set $Z = \mat{\tau}{z}{z}{\tau'} \in \h{2}$, $\tau, \tau' \in \h{1}$, $z \in \cmplx$, $\IM(z)^2 < \IM(\tau)\IM(\tau')$, and $T = \mat{n}{r/2}{r/2}{m}.$  The space of Maass spezialchars is defined as the set of Siegel cuspforms whose Fourier coefficients satisfy 
% \begin{equation*}
% A_{F}(n,r,m) = \sum_{d \mid \gcd(n,r,m)} d^{k-1} A_{F}\left( \frac{mn}{d^2}, \frac{r}{d},1\right)
% \end{equation*}
% for every $m,n,r \in \inte$ with $m,n,4mn-r^2 \geq 0$. 
\section{Congruences between Saito-Kurokawa lifts and non-Maass forms}\label{sec:congruences}

% In this section we will use the results from section \ref{sec:eseries} as well as arguments as used in \cite{jimbrown} to 
% show that under certain conditions one has a congruence between a Saito-Kurokawa lift and a non-CAP form.  We follow the same general outline as in \cite{jimbrown}, improving the methods where possible. As always, we assume $k>3$ throughout this
% section.
In this section we will use the results obtained in section \ref{sec:eseries} as well as arguments originally used in \cite{jimbrown} to produce a congruence between a Saito-Kurokawa lift $\Ff$ and a form $G \in \mathcal{S}_{k}^{\NM}(\Sp_4(\inte))$.  This congruence will provide the foundation for providing the lower bound on the size of the appropriate Selmer group in section \ref{sec:finalresult}.

Let $k>3$ be an integer and $p$ a prime so that $p > 2k-2$ and $\gcd(p,N)=1$.  Let $E$ be a finite extension of $\rat_{p}$ with ring of integers $\ROI$ and uniformizer $\varpi$.  We fix an embedding of $E$ into $\cmplx$ that is compatible with the embeddings fixed in section \ref{sec:notation}.  Let $\p$ be the prime of $\ROI$ lying over $p$. 

Given two Siegel modular forms $F$ and $G$, we write $F \equiv G (\modu \varpi^{m})$ to mean that $\ord_{\varpi}(A_{F}(T) - A_{G}(T)) \geq m$ for all $T$.  If $F$ and $G$ are Hecke eigenforms, we denote a congruence between the eigenvalues as $F \equiv_{\ev} G (\modu \varpi^{m})$. 

Let $f \in S_{2k-2}(\SL_2(\inte))$ be a newform and $\Ff$ the Saito-Kurokawa lift of $f$.  As in \cite{jimbrown}, we begin by replacing the normalized Eisenstein series $D_{E\Slash{\rho}}(\diag[Z,W], (5-k)/2)$ by its trace so as to work with level 1.  Write
\begin{equation*}
\mathcal{E}(Z,W) = \sum_{\gamma \times \delta \in (\Sp_4(\inte) \times \Sp_4(\inte))/(\Gamma_0^4(N) \times \Gamma_0^4(N))} (D_{E\Slash{\rho}})\Slash{\gamma \times \delta}(\diag[Z,W],(5-k)/2).
\end{equation*}
We then have that $\mathcal{E}(Z,W)$ is a Siegel cusp form of weight $k$ and level 1 in each variable separately.  In addition, one sees that the coefficients of $\mathcal{E}(Z,W)$ remain in $\intep[\chi]$ by an application of the $q$-expansion principle for Siegel modular forms (\cite{chaifaltings}, Proposition 1.5). 

Let $f_0 = f, f_1, \dots, f_{m}$ be an orthogonal basis of newforms for $S_{2k-2}(\SL_2(\inte))$ and let $F_0=\Ff, F_1 = F_{f_{1}}, \dots, F_{m}= F_{f_{m}}, F_{m+1}, \dots, F_{r}$ be an orthogonal basis of eigenforms of $\mathcal{S}_{k}(\Sp_4(\inte))$.  Note we are implicitly using Corollary \ref{thm:skcomplexconjugation} and Theorem \ref{thm:innerproducts} here.  We enlarge $E$ here if necessary so that both of the bases are defined over $\ROI$ and $\ROI$ contains the values of $\chi$.  We write 
\begin{equation*}
\mathcal{E}(Z,W) = \sum_{i,j} c_{i,j} F_{i}(Z)F_{j}^{c}(W)
\end{equation*}
for some $c_{i,j} \in \cmplx$.  The reader is urged to see (\cite{shimura83}, Lemma 1.1) or \cite{garrett2} for the general principles of such expansions.  In \cite{jimbrown} the assumption was made that $f$ have real Fourier coefficients to remove the complex conjugation from the formulas. Though this assumption will eventually be necessary, we do not immediately make it.  This adds some minor complexity to the argument, but provides a much better picture of why the assumption is necessary for the method and is not just one made out of convenience.

\begin{proposition}\label{prop:diagonalizeeseries} With the notation as above, we have
\begin{equation}\label{eqn:eseriesexpansion}
\mathcal{E}(Z,W) = \sum_{0 \leq i \leq r} c_{i,i} F_{i}(Z) F_{i}^{c}(W),
\end{equation}
i.e., $c_{i,j} = 0$ unless $i = j$.  In particular, one has 
\begin{equation*}
c_{i,i} = \frac{\mathcal{A}_{k,N} L^{\Sigma}_{\st}(5-k,F_{i}, \chi)}{\pi^3 \langle F_{i}^{c}, F_{i}^{c} \rangle}.
\end{equation*}
\end{proposition}

\begin{proof}  Fix $j_0$ with $0 \leq j_0 \leq r$ and consider the inner product $\langle \mathcal{E}(Z,W), F^{c}_{j_0}(W) \rangle$.  Using the orthogonality of the basis we have
\begin{equation}\label{eqn:diag1}
\langle \mathcal{E}(Z,W), F^{c}_{j_0}(W) \rangle = \sum_{0 \leq i \leq r} c_{i,j_0} \langle F_{j_0}^{c}, F_{j_0}^{c} \rangle F_{i}(Z).
\end{equation}
On the other hand, combining the facts that $$\langle \mathcal{E}(Z,W), \Ff^{c}(W) \rangle_{\Sp_4(\inte)} = \langle D_{E\Slash{\rho}}(\diag[Z,W], (5-k)/2), \Ff^{c}(W) \rangle_{\Gamma_0^4(N)}$$ and $\Ff\Slash{\varsigma} = \Ff$ since $\Ff$ is of full level with equation (\ref{eqn:shimuraeq1}) gives 
\begin{equation*}
\langle \mathcal{E}(Z,W), F^{c}_{j_0}(W) \rangle = \pi^{-3} \,\mathcal{A}_{k,N}
L^{\Sigma}_{\st}(5-k,F_{j_0}, \chi) F_{j_0}(Z).
\end{equation*}
Combining this with equation (\ref{eqn:diag1}) and using that the $F_{i}$ form a basis allows us to conclude that $c_{i,j_0} = 0$ unless $i = j_0$ and if $i = j_0$ we have
\begin{equation*}
c_{j_0,j_0} = \frac{\mathcal{A}_{k,N} L_{\st}^{\Sigma}(5-k,F_{j_0},\chi)}{\pi^3 \langle F_{j_0}^c, F_{j_0}^{c} \rangle}.
\end{equation*}
 Since $j_0$ was arbitrary, we have
\begin{equation*}
\mathcal{E}(Z,W) = \sum_{0 \leq i \leq r} c_{i,i} F_{i}(Z) F_{i}^{c}(W)
\end{equation*}
with the $c_{i,i}$ as desired.
\end{proof}

As our goal is to produce a congruence between $\Ff$ and a Siegel cusp form that is not a CAP form, the first step is a theorem which
allows us to kill all of $F_{i}$ with $1 \leq i \leq m$.  First we need the following result giving the existence of complex periods.

\begin{theorem} (\cite{shimuraperiods}, Theorem 1) Let $f \in S_{2k-2}(\SL_2(\inte), \ROI)$ be a newform.  There exist complex periods $\Omega_{f}^{\pm}$ such that for each integer $m$ with $0 < m < 2k-2$ and every Dirichlet character $\chi$ one has 
\begin{equation*}
\frac{L(m,f,\chi)}{\tau(\chi)(2\pi i)^{m}} \in \left\{ \begin{array}{ll} \Omega_{f}^{-} \ROI_{\chi} & \text{if $\chi(-1) = (-1)^{m}$,} \\
                                \Omega_{f}^{+} \ROI_{\chi} & \text{if $\chi(-1) = (-1)^{m-1}$,} \end{array} \right.
\end{equation*}
where $\tau(\chi)$ is the Gauss sum of $\chi$ and $\ROI_{\chi}$ is the extension of $\ROI$ generated by the values of $\chi$.  We write $L_{\alg}(m,f,\chi)$ to denote the value $\frac{L(m,f,\chi)}{\tau(\chi)(2\pi i)^{m}}$. 
\end{theorem}

The theorem we will use to kill the $F_{i}$ for $1 \leq i \leq m$ is as follows. 

\begin{theorem} (\cite{jimbrown}, Theorem 5.4) Let $f=f_0, f_1, \dots, f_{m}$ be a basis of newforms of $S_{2k-2}(\SL_2(\inte),\ROI)$ with $k>2$.  Suppose that the residual Galois representation $\overline{\rho}_{f}$ is irreducible and $f$ is ordinary at $\p$.  There exists a Hecke operator $t \in \T_{\ROI}$ such that 
\begin{equation*}
tf_{i} = \left\{ \begin{array}{ll} \alpha f & \text{if $i=0$,}\\
        0 & \text{if $1 \leq i \leq m$,} \end{array} \right.
\end{equation*}
where $\alpha = \frac{u \langle f, f \rangle}{\Omega_{f}^{+} \Omega_{f}^{-}}$, $u$ is a unit in $\ROI$.
\end{theorem}
 
From this point on we assume that $f$ is ordinary at $\p$ and $\overline{\rho}_{f}$ is irreducible so we are able to apply this theorem.  We write $F\equiv G (\modu \varpi^{n})$ to indicate that $\ord_{\varpi}(A_{F}(T) - A_{G}(T)) \geq n$ for all $T$, i.e., the congruence is a congruence of Fourier coefficients.

Applying the fact that the Saito-Kurokawa correspondence is Hecke-equivariant we are able to conclude there exists $t^{S} \in \T_{\ROI}^{S}$ so that 
\begin{equation*}
t^{S} F_{i} = \left\{ \begin{array}{ll} \alpha F_{0} & \text{if $i = 0$,}\\
                                        0 & \text{if $1 \leq i \leq m$.} \end{array} \right.
\end{equation*}
Applying $t^{S}$ to equation (\ref{eqn:eseriesexpansion}) we obtain

\begin{equation}\label{eqn:eseriesheckeopexpansion}
t^{S} \mathcal{E}(Z,W) = \alpha c_{0,0} F_{0}(Z)F_{0}^{c}(W) + \sum_{m<j\leq r} c_{j,j} t^{S}F_{j}(Z) F_{j}^{c}(W).
\end{equation}
Before we study $\ord_{\varpi}(\alpha c_{0,0})$, we show how it ``controls'' a congruence between $F_{0} = \Ff$ and a non-CAP eigenform.  

Suppose we can have $\ord_{\varpi}(\alpha c_{0,0}) = M < 0$, i.e., there exists a $\varpi$-unit $\beta$ so that $\alpha c_{0,0} = \varpi^{M} \beta$.  Corollary \ref{thm:skfourier} gives that there exists a $T_0$ so that $\varpi \nmid A_{F_0^{c}}(T_0)$.  We claim that this implies that $c_{i,i} \neq 0$ for at least one $i$ with $i > 0$.  If not, we would have 
\begin{equation*}
t^{S} \mathcal{E}(Z,W) = \varpi^{M} \beta F_{0}(Z)F_{0}^{c}(W).
\end{equation*}
However, since $\mathcal{E}(Z,W)$ has $\varpi$-integral Fourier coefficients, $t^{S}\mathcal{E}(Z,W)$ does as well. Thus, upon multiplying both sides of the equation by $\varpi^{-M}$ we obtain that $F_0(Z)F_0^{c}(W) \equiv 0(\modu \varpi)$, clearly a contradiction.  

We now expand each side of equation (\ref{eqn:eseriesheckeopexpansion}) in terms of $W$, reduce modulo $\varpi$, and equate $T_0^{\text{th}}$ Fourier coefficients.  The $\ROI$-integrality of the Fourier coefficients of $t^{S} \mathcal{E}(Z,W)$ combined with the fact that $M < 0$ gives 
\begin{equation*}
F_0(Z) \equiv - \frac{\varpi^{-M}}{A_{F_{0}^{c}}(T_0) \beta} \sum_{m < j \leq r} c_{j,j} A_{F_{j}^{c}(T_0)} t^{S} F_{j}(Z) (\modu \varpi^{-M}).
\end{equation*}
Set 
\begin{equation*}
G = - \frac{\varpi^{-M}}{A_{F_{0}^{c}}(T_0) \beta} \sum_{m < j \leq r} c_{j,j} A_{F_{j}^{c}(T_0)} t^{S} F_{j}(Z).
\end{equation*}
First, we note that if $G \equiv 0 (\modu \varpi^{-M})$ we obtain a contradiction to Corollary \ref{thm:skfourier}, so we must have a nontrivial congruence.  Secondly, since all of the $F_{j}$'s with $m < j \leq r$ lie in $\mathcal{S}_{k}^{\NM}(\Sp_4(\inte))$ and as we noted earlier this space is stable under Hecke operators, we have that $t^{S} F_{j}$ is in $\mathcal{S}_{k}^{\NM}(\Sp_4(\inte))$ for $m < j \leq r$.  As $G$ is a linear combination of forms in $\mathcal{S}_{k}^{\NM}(\Sp_4(\inte))$, we must have $G \in \mathcal{S}_{k}^{\NM}(\Sp_4(\inte))$.  Thus, if we show that $\ord_{\varpi}(\alpha c_{0,0})$ is less then 0, we have a nontrivial congruence between $\Ff$ and a form in $S_{k}^{\NM}(\Sp_4(\inte))$, in particular, between a Saito-Kurokawa lift and a non-Saito-Kurokawa lift.  

Our discussion of $\alpha c_{0,0}$ will be brief as most of the work done in \cite{jimbrown} applies verbatim to this situation.  Corollary \ref{thm:skcomplexconjugation} allows us to apply Theorem \ref{thm:innerproducts} to $\langle F_0^{c}, F_{0}^{c} \rangle$ as this is $\langle F_{f^{c}}, F_{f^{c}}\rangle$.  Following the same argument as in \cite{jimbrown} we are able to conclude that 
\begin{equation*}
\alpha c_{0,0} = \mathcal{B}_{k,N,D,\chi} \mathcal{L}(k,f,D, \chi) \frac{\langle f, f \rangle}{\langle f^{c}, f^{c} \rangle}
\end{equation*}
where 
\begin{equation*}
\mathcal{L}(k,f,D,\chi) = \frac{L^{\Sigma}(3-k,\chi) L_{\alg}(k-1,f^{c},\chi_{D}) L_{\alg}(1,f,\chi) L_{\alg}(2,f,\chi)}{L_{\alg}(k,f^{c})}
\end{equation*}
and 
\begin{equation*}
\mathcal{B}_{k,N,D,\chi} = \frac{(-1)^{k+1} 2^{2k+4} 3 u |D|^{k} \tau(\chi_{D}) \tau(\chi)^2}{(k-1) [\Sp_2(\inte): \Gamma_0^2(N)] |D|^{3/2} |c_{g_{f^{c}}}(|D|)|^2 L_{\Sigma}(1,f,\chi)L_{\Sigma}(2,f,\chi)}
\end{equation*}
where $D$ is a fundamental discriminant so that $\chi_{D}(-1) = -1$.  It was shown in \cite{jimbrown} that as long as $\gcd(p,D[\Sp_4(\inte):\Gamma_0^4(N)]) = 1$ one has $\ord_{\varpi}(\mathcal{B}_{k,N,D,\chi}) \leq 0$.  Thus, we would be able to get a divisibility condition on $L$-functions associated to $f$ guaranteeing the existence of the congruence if not for the factor $\frac{\langle f, f \rangle}{\langle f^{c}, f^{c} \rangle} = \frac{\langle f, f \rangle}{\langle f, f \rangle ^{c}}$ (\cite{shimura1976}, Lemma 4).  Unfortunately, we do not even know in general if this quantity is algebraic!  However, if we assume that $f$ has Fourier coefficients lying in $\real$, then this term is 1 and so does not cause any further issues. In that case, we have the following theorem.

\begin{theorem}\label{thm:noneigenformcongruence} Let $k > 3$ be an integer and $p$ a prime so that $p > 2k-2$.  Let $f \in S_{2k-2}(\SL_2(\inte),\ROI)$  be a newform with real Fourier coefficients and $\Ff$ the Saito-Kurokawa lift of $f$.  Let $f$ be ordinary at $p$ and $\overline{\rho}_{f}$ be irreducible.  If there exists $N>1$, a fundamental discriminant $D$ so that $(-1)^{k-1}D>0$, $\chi_{D}(-1)=-1$, $p \nmid ND[\Sp_4(\inte):\Gamma_0^4(N)]$, and a Dirichlet character $\chi$ of conductor $N$ so that 
\begin{equation*}
-M=\ord_{\varpi}(\mathcal{L}(k,f,D,\chi)) <0
\end{equation*}
then there exists $G \in \mathcal{S}_{k}^{\NM}(\Sp_4(\inte))$ so that 
\begin{equation*}
\Ff \equiv G (\modu \varpi^{M}).
\end{equation*}
\end{theorem}

We repeat Lemma 6.4 of \cite{jimbrown} here for completeness, though it will not be needed.

\begin{lemma} With the set-up as in Theorem \ref{thm:noneigenformcongruence}, there exists an eigenform $F \in \mathcal{S}_{k}^{\NM}(\Sp_4(\inte))$ so that $\Ff \equiv_{\ev} F(\modu \varpi)$. 
\end{lemma}

\begin{proof}  Let $G$ be as in Theorem \ref{thm:noneigenformcongruence}.  The Hecke algebra $\T^{S}_{\ROI}$ decomposes as 
\begin{equation*}
\T_{\ROI}^{S} \cong \prod \T^{S}_{\ROI,\mathfrak{m}}
\end{equation*}
where the product is over the maximal ideals of $\T_{\ROI}^{S}$ containing $\varpi$.  Write $\mathfrak{m}_{\Ff}$ for the maximal ideal corresponding to $\Ff$, i.e., $\mathfrak{m}_{\Ff}$ is the kernel of the $\ROI$-algebra homomorphism $\lambda_{\Ff}: \T_{\ROI}^{S} \rightarrow \ROI$ sending $t$ to the eigenvalue of $t$ acting on $\Ff$.  The decomposition gives that there exists a Hecke operator $t \in \T_{\ROI}^{S}$ so that $t \Ff = \Ff$ and $t F = 0$ if $F$ does not correspond to the ideal $
\mathfrak{m}_{\Ff}$, i.e., if $F \not \equiv_{\ev} \Ff (\modu \varpi)$.  

Write $G = \sum c_{i} F_{i}$.  By construction, the only $F_{i}$ that appear are in $\mathcal{S}_{k}^{\NM}(\Sp_4(\inte))$.  Applying the Hecke operator $t$ to $G$ we see that $\Ff \equiv t G (\modu \varpi)$ and so $tG \not \equiv 0 (\modu \varpi)$.  Thus, there is an eigenform $F \in \mathcal{S}_{k}^{\NM}(\Sp_4(\inte))$ so that $\Ff \equiv_{\ev} F (\modu \varpi)$ as claimed.
\end{proof}

We close this section by phrasing our results in terms of the CAP ideal (\cite{klosin}).  This ideal can be thought of as a generalization of the Eisenstein ideal to our situation.  As the Eisenstein ideal measures congruences be to Eisenstein series, the CAP ideal measures congruences between $\Ff$ and non-CAP modular forms.

Let $\T_{\ROI}^{\NM}$ denote the image of $\T_{\ROI}^{S}$ inside of $\End_{\cmplx}(\mathcal{S}_{k}^{\NM}(\Sp_4(\inte)))$.  Let $\phi: \T_{\ROI}^{S} \rightarrow \T_{\ROI}^{\NM}$ denote the canonical $\ROI$-algebra surjection.  Denote the annihilator of $\Ff$ in $\T_{\ROI}^{S}$ by $\Ann(\Ff)$.  The annihilator of $\Ff$ in $\T_{\ROI}^{S}$ is a prime ideal and one has that the map $\lambda_{\Ff}$ induces an $\ROI$-algebra isomorphism 
\begin{equation*}
\T_{\ROI}^{S}/\Ann(\Ff) \cong \ROI.
\end{equation*}
As the map $\phi$ is surjective, one has that $\phi(\Ann(\Ff))$ is an ideal in $\T_{\ROI}^{\NM}$.  We call this ideal the {\it CAP ideal associated to $\Ff$}.  

One has that there exists an $r \in \inte_{\geq 0}$ so that the following diagram commutes:

\begin{figure}[h!]
\centerline{\xymatrix{ \T_{\ROI}^{S} \ar^{\phi}[r] \ar[d] & \T_{\ROI}^{\NM} \ar[d] \\ \T_{\ROI}^{S}/\Ann(\Ff) \ar^{\phi}[r] \ar^{\lambda_{\Ff}}[d] & \T_{\ROI}^{\NM}/\phi(\Ann(\Ff)) \ar^{\simeq}[d] \\ \ROI \ar[r] & \ROI/\varpi^{r} \ROI.}}
\end{figure}

\noindent Note that all of the maps in the above diagram are $\ROI$-algebra surjections.

\begin{corollary}\label{corl:capideal} With $r$ as in the diagram above and $M$ as in Theorem \ref{thm:noneigenformcongruence} we have $r \geq M$.
\end{corollary}

\begin{proof}  Let $G$ be as in Theorem \ref{thm:noneigenformcongruence}.  Choose a $t \in \phi^{-1}(\varpi^{r}) \subset \T_{\ROI}^{S}$.  Note that this means that $t G = \varpi^{r} G$.  We also have by the commutativity of the diagram that $t \in \Ann(\Ff)$ and so the congruence in Theorem \ref{thm:noneigenformcongruence} gives $\varpi^{r} G \equiv 0 (\modu \varpi^{M})$, i.e., 
\begin{equation*}
G \equiv 0 (\modu \varpi^{M-r}).
\end{equation*}
Now assume that $r < M$. The fact that $\Ff \equiv G (\modu \varpi^{M})$ and $G \equiv 0 (\modu \varpi)$ imply that 
\begin{equation*}
\Ff \equiv 0 (\modu \varpi).
\end{equation*}
However, this is impossible and so we must have $r \geq M$ as claimed. 
\end{proof}

% We now only need to rule out the possibility that $\Ff$ is congruent to a CAP form that is not a Saito-Kurokawa lift.  To see it is not possible for $\Ff$ to be congruent to a CAP form $H$ that is not a Saito-Kurokawa lift, we use that necessarily the Galois representation of $H$ would have to split up into 2 irreducible components or 4 characters. The first case is not possible since  neither of which can occur since $\ov{\rho}$

\section{Residually reducible representations and lattices}\label{sec:lattices}

In this section we construct the lattice we will use in section \ref{sec:finalresult} to give a lower bound on the appropriate Selmer group.  Our main result here is a generalization of Theorem 1.1 of \cite{urbanduke}.  Where his result deals with the case of $\ov{\rho}^{\semi}$ splitting into two irreducible representations, our results are for when $\ov{\rho}^{\semi}$ splits into three irreducible representations.  The proofs given here are adapted from those in \cite{urbanduke} to our situation.  As Theorem 1.1 of \cite{urbanduke} is a variant of the main result of \cite{urban99}, our result here is very much a variant of the main result of \cite{jb08}.  The results contained in this section or some variant of them will also appear in \cite{skinnerurbangl2}.

Let $\ROI$ be a discrete valuation ring and $\mc{R}$ a reduced local commutative algebra that is finite over $\ROI$ and is Henselian.  Let $\mf{m}_{\mc{R}}$ be the maximal ideal of $\mc{R}$ and $\kappa_{\mc{R}}$ the residue field.  We use the fact that $\mc{R}$ is reduced to embed it into a product of its irreducible components
\begin{equation*}
\mc{R} \subset \widetilde{\mc{R}} = \prod_{\wp} \ROI_{\wp} \subset \prod_{\wp} F_{\wp} = F_{\mc{R}}
\end{equation*}
where $F_{\wp}$ is the field of fractions of $\ROI_{\wp}$.  Let $\mf{m}_{\wp}$ be the maximal ideal of $\ROI_{\wp}$ and $\kappa_{\wp}$ the residue field.  

Let $M_{F_{\mc{R}}}$ be a finite free $F_{\mc{R}}$-module and let $\lat$ be a $\mc{R}$-submodule of $M_{F_{\mc{R}}}$.  We say $\lat$ is a $\mc{R}$-lattice if $\lat$ is finite over $\mc{R}$ and $\lat \otimes_{\mc{R}} F_{\mc{R}} = M_{F_{\mc{R}}}$.

\begin{theorem}\label{thm:lattices}  Let $\mc{A}$ be a $\mc{R}$-algebra and let $\rho$ be an absolutely irreducible representation of $\mc{A}$ on $F_{\mc{R}}^{n}$. Suppose there are at least $n$ distinct elements in $\kappa_{\mc{R}}^{\times}$ and suppose that there exist three representations $\rho_{i}$ for $1 \leq i \leq 3$ in $\M_{n_{i}}(\mc{R})$ and $I \subset \mc{R}$ a proper ideal of $\mc{R}$ such that
\begin{enumerate}
\item the coefficients of the characteristic polynomial of $\rho$ belong to $\mc{R}$;
\item the characteristic polynomials of $\rho$ and $\rho_1 \oplus \rho_2 \oplus \rho_3$ are congruent modulo $I$;
\item the residual representations $\ov{\rho}_{i}$ are absolutely irreducible for $1 \leq i \leq 3$;
\item $\ov{\rho}_{i} \not \equiv \ov{\rho}_{j}$ if $i \neq j$.
\end{enumerate}
The there exists an $\mc{A}$-stable $\mc{R}$-lattice $\lat$ in $F_{\mc{R}}^{n}$ and $\mc{R}$-lattices $\mc{T}_1, \mc{T}_2$ of $F_{\mc{R}}$ such that we have the following exact sequence of $\mc{A}$-modules:\\

\begin{figure}[h!]
\centerline{\xymatrix{ 0 \ar[r] & \mc{N}_1 \oplus \mc{N}_2 \ar[r] & \lat \otimes \mc{R}/I \ar[r]  &\rho_3 \otimes \mc{R}/I \ar@/_1pc/[l]^{s} \ar[r] & 0,
}}
\end{figure}

\noindent where $\mc{N}_{i} = \rho_{i} \otimes \mc{T}_{i}/I\mc{T}_{i}$ for $i=1,2$ and $s$ is a section (though only of $\mc{R}/I$-modules.)  Moreover, $\lat$ has no quotient isomorphic to a representation $\ov{\rho}'$ with $\ov{\rho}'^{\semi} = \ov{\rho}_1 \oplus \ov{\rho}_2$. 
\end{theorem}

Let $\rho_{\wp}: \mc{A} \rightarrow \M_{n}(F_{\wp})$ be the representation arising from $\rho$ via the projection $F_{\mc{R}} \rightarrow F_{\wp}$. Since the trace of $\rho_{\wp}$ takes values in $\ROI_{\wp}$, we can find an $\mc{A}$-stable $\ROI_{\wp}$-lattice $\lat_{\wp}$ in $F_{\wp}^{n}$.  Write $\ov{\rho}_{\wp}$ for the residual representation $\ov{\rho}_{\wp}: \mc{A} \rightarrow \M_{n}(\kappa_{\wp})$.  We abuse notation and write $\ov{\rho}_{i}$ for the residual representation $\ov{\rho}_{i} : \mc{A} \rightarrow \M_{n_{i}}(\kappa_{\mc{R}})$ as well as the residual representation arising from the natural projection from $\kappa_{\mc{R}} \rightarrow \kappa_{\wp}$. The fact that $\ov{\rho}_{\wp}^{\semi} = \ov{\rho}_{1} \oplus \ov{\rho}_{2} \oplus \ov{\rho}_{3}$ allows us to conclude that there is an $\ROI_{\wp}$-basis of $\lat_{\wp}$ so that 
\begin{equation}
\rho_{\wp}: \mc{A} \rightarrow \M_{n}(\ROI_{\wp})
\end{equation}
with 
\begin{equation}\label{eqn:urbanstar}
\ov{\rho}_{\wp}(a) = \begin{pmatrix} \ov{\rho}_{1}(a) & \star_1 & \star_2 \\ \star_3 & \ov{\rho}_{2}(a) & \star_4 \\ 0 & 0 & \ov{\rho}_{3}(a) \end{pmatrix}.
\end{equation}
Moreover, we have $\star_1 \star_3 = 0$.  Set $\rho_{\widetilde{\mc{R}}} = (\rho_{\wp})_{\wp} : \mc{A} \rightarrow \M_{n}(\widetilde{R})$.  

The fact that each $\ov{\rho}_{i}$ is absolutely irreducible implies that $\text{im} \ov{\rho}_{i} = \M_{n_{i}}(\kappa_{\mc{R}})$.  Combining this with $\ov{\rho}_{i} \not \equiv \ov{\rho}_{j}$ if $i \neq j$ allows us to conclude that there exists $a_0 \in \mc{A}$ so that $\det(X-\ov{\rho}_{\widetilde{\mc{R}}}(a_0))$ has $n$ distinct roots $\ov{\alpha}_1, \dots, \ov{\alpha}_{n}$ where $\ov{\rho}_{\widetilde{\mc{R}}}(a) \in \M_{n}\left(\frac{\widetilde{R}}{\Rad(\widetilde{\mc{R}})}\right)$.  Hensel's lemma guarantees that there exists $n$ distinct elements $\alpha_1, \dots, \alpha_{n} \in \mc{R}$ that are roots of $\det(X - \rho_{\widetilde{\mc{R}}}(a_0))$ so that $\alpha_{i} \equiv \ov{\alpha}_{i} (\modu \mf{m}_{\mc{R}})$.  After a change of basis we may assume that 
\begin{equation*}
\rho_{\widetilde{\mc{R}}}(a_0) = \diag(\alpha_1, \dots, \alpha_{n}).
\end{equation*}
Lemma 2.1 of \cite{jb08} gives that the $\mc{R}$-submodule of $\M_{n}(\widetilde{\mc{R}})$ generated by the powers of $\rho_{\widetilde{\mc{R}}}(a_0)$ is exactly the set of diagonal matrices with entries in $\mc{R}$.  So for all $i$ with $1 \leq i \leq n$ there exists $f_{i} \in \mc{A}$ so that 
\begin{equation*}
\rho_{\widetilde{\mc{R}}}(f_{i}) = E(\alpha_{i}) = \diag(0, \dots, 0, 1, 0, \dots 0)
\end{equation*}
with $1$ in the $i$th place. We also set $a_1, a_2,$ and $a_3$ to be the elements in $\mc{A}$ so that
\begin{align*}
\rho_{\widetilde{\mc{R}}}(a_{1}) &= E_1 = \begin{pmatrix} 1_{n_1} & 0 & 0 \\ 0 & 0_{n_2} & 0 \\ 0 & 0 & 0_{n_3} \end{pmatrix},\\
\rho_{\widetilde{\mc{R}}}(a_2) &= E_2 = \begin{pmatrix} 0_{n_1} & 0 & 0 \\ 0 & 1_{n_2} & 0 \\ 0 & 0 & 0_{n_3} \end{pmatrix},\\
\end{align*}
and 
\begin{equation*}
\rho_{\widetilde{\mc{R}}}(a_3) = E_3 = \begin{pmatrix} 0_{n_1} & 0 & 0 \\ 0 & 0_{n_2} & 0 \\ 0 & 0 & 1_{n_3} \end{pmatrix}.
\end{equation*}

We now consider $\widetilde{\mc{R}}^{n}$ with an action of $\mc{A}$ given by $\rho_{\widetilde{\mc{R}}}$.  Let $(e_1, \dots, e_{n})$ be the canonical basis and set $\lat$ as the $\mc{R}$-sublattice of $\widetilde{\mc{R}}^{n}$ generated by $\rho_{\widetilde{\mc{R}}}(a) e_{n}$ as $a$ runs through $\mc{A}$.  It is clear that $\lat$ is $\mc{A}$-stable by construction.  Set $\lat_{i} = E_{i}(\lat)$ for $i = 1, 2, 3$. Trivially we have $\lat = \lat_1 \oplus \lat_2 \oplus \lat_3$.  

\begin{lemma} The lattice $\lat_{i}$ is free of rank $n_{i}$ over $\mc{R}$ for $1 \leq i \leq 3$.
\end{lemma}

\begin{proof}  This is Lemma 1.1 of \cite{urbanduke}.  We give the proof here for $i=3$ as the other two cases are identical. The definition of $\lat_3$ gives that $\lat_3 \otimes_{\mc{R}} \kappa_{\mc{R}}$ is generated by
\begin{equation*}
\begin{pmatrix} 0_{n_1} & 0 & 0 \\ 0 & 0_{n_2} & 0 \\ 0 & 0 & \ov{\rho}_{3}(a) \end{pmatrix} \ov{e}_{n}
\end{equation*}
as $a$ runs through $\mc{A}$.  The fact that $\text{im} \ov{\rho}_{3} = \M_{n_3}(\kappa_{\mc{R}})$ gives that $\lat_3 \otimes_{\mc{R}} \kappa_{\mc{R}} = \kappa_{\mc{R}} \ov{e}_{n_2 +1} \oplus \cdots \oplus \kappa_{\mc{R}} \ov{e}_{n}$.  For $n_2 +1 \leq i \leq n$ let $e_{i}' \in \lat_{3}$ be a lifting of $\ov{e}_{i}$.  Theorem 2.3 in \cite{matsumura} (essentially Nakayama's lemma) gives that $\lat_3 = \mc{R} e_{n_1+1}' + \cdots + \mc{R} e_{n}'$.  We can now use the fact that $\lat_3 \otimes_{\mc{R}} F_{\mc{R}}$ is of rank $n_3 = n - n_1 - n_2$ over $F_{\mc{R}}$ to conclude the sum must be direct.
\end{proof}

Write $\rho_{\lat} = \rho_{\widetilde{\mc{R}}}\mid_{\lat}$.  The decomposition $\lat = \lat_1 \oplus \lat_2 \oplus \lat_3$ allows us to write $\rho_{\lat}$ in blocks
\begin{equation*}
\rho_{\lat}(a) = \begin{pmatrix} A^{1,1}(a) & A^{1,2}(a) & A^{1,3}(a) \\ A^{2,1}(a) & A^{2,2}(a) & A^{2,3}(a) \\
                A^{3,1}(a) & A^{3,2}(a) & A^{3,3}(a)
                \end{pmatrix}
\end{equation*}
where 
\begin{equation*}
A^{i,j}(a) = \Res_{\lat_{j}}(E_{i} \circ \rho_{\widetilde{\mc{R}}}(a)) \in \Hom_{\mc{R}}(\lat_{j}, \lat_{i}).
\end{equation*}

\begin{lemma}\label{lem:surj} The map 
\begin{align*}
\mc{A} &\rightarrow \Hom_{\mc{R}}(\lat_3, \lat_{1}) \oplus \Hom_{\mc{R}}(\lat_3, \lat_{2}) \cong \Hom_{\mc{R}}(\lat_3, \lat_1 \oplus \lat_2)\\
a &\mapsto (A^{1,3}(a), A^{2,3}(a))
\end{align*}
is surjective as a map of $\mc{R}$-modules.  Moreover, if there exists $a \in \mc{A}$ so that $A^{1,2}(a) \neq 0$, then the map 
\begin{align*}
\mc{A} &\rightarrow \Hom_{\mc{R}}(\lat_2, \lat_1) \oplus \Hom_{\mc{R}}(\lat_3, \lat_{1}) \oplus \Hom_{\mc{R}}(\lat_3, \lat_{2}) \\
a &\mapsto (A^{1,2}(a), A^{1,3}(a), A^{2,3}(a))
\end{align*}
is surjective as a map of $\mc{R}$-modules. 
\end{lemma}

\begin{proof} This is a generalization of Lemma 1.2 of \cite{urbanduke}.  The necessary work was done in Lemma 2.3 of \cite{jb08}, so here we merely reduce to what was shown there.   

We prove the first statement.  If we can show that the map 
\begin{align*}
\mc{A}\otimes_{\mc{R}} \kappa_{\mc{R}} &\overset{\vartheta}{\rightarrow} \Hom_{\kappa_{\mc{R}}}(\lat_3 \otimes_{\mc{R}} \kappa_{\mc{R}}, \lat_1 \otimes_{\mc{R}} \kappa_{\mc{R}})\oplus \Hom_{\kappa_{\mc{R}}}(\lat_3 \otimes_{\mc{R}} \kappa_{\mc{R}}, \lat_2 \otimes_{\mc{R}} \kappa_{\mc{R}})\\
 &\cong \bigoplus_{i=n_2+1}^{n} \Hom_{\kappa_{\mc{R}}}(\kappa_{\mc{R}} \ov{e}_{i} , \lat_{1}\otimes_{\mc{R}} \kappa_{\mc{R}}) \oplus \bigoplus_{k=n_2+1}^{n} \Hom_{\kappa_{\mc{R}}}(\kappa_{\mc{R}} \ov{e}_{k} , \lat_{2}\otimes_{\mc{R}} \kappa_{\mc{R}}) \\
&\cong \bigoplus_{i=n_2+1}^{n} \bigoplus_{j=1}^{n_1} \Hom_{\kappa_{\mc{R}}}(\kappa_{\mc{R}} \ov{e}_{i} , \kappa_{\mc{R}} \ov{e}_{j}) \oplus \bigoplus_{k=n_2+1}^{n} \bigoplus_{l=n_1+1}^{n_2} \Hom_{\kappa_{\mc{R}}}(\kappa_{\mc{R}} \ov{e}_{k} , \kappa_{\mc{R}} \ov{e}_{l})
\end{align*}
is surjective, then Nakayama's lemma will give the result.  However, to see surjectivity here it is enough to see that the image contains each $(i,j) \times (k,l)$-factor.  In terms of the block decomposition into matrices as above, this amounts to showing that $e_{u,v}$ is in the image for all $1 \leq u \leq n_2$, $n_2+1 \leq v \leq n$ for $e_{u,v}$ the matrix with a $1$ in the $uv$th entry and $0$'s elsewhere.  That this is true is shown in Lemma 2.3 of \cite{jb08}. 

The second statement follows from the same method. 

%Let $(\phi, \psi) \in \Hom_{\kappa_{\mc{R}}}(\kappa_{\mc{R}} \ov{e}_{i} , \lat_{1}\otimes_{\mc{R}} \kappa_{\mc{R}}) \times \Hom_{\kappa_{\mc{R}}}(\kappa_{\mc{R}} \ov{e}_{j} , \lat_{2}\otimes_{\mc{R}} \kappa_{\mc{R}})$.  Choose $a \in \mc{A}$ so that $\vartheta(a)(\ov{e}_{n}) = (\phi(\ov{e}_{i}), \psi(\ov{e}_{j}))$ which is possible by our definition of $\lat$. The surjectivity of $\ov{\rho}_{3}$ gives $t_{i} \in \mc{A}$ so that $\ov{\rho}_3(t_{i}) \ov{e}_{i} = \ov{e}_{n}$ and similarly for $t_{j}$. Note that $\ov{\rho}_{\widetilde{\mc{R}}}(a_3 t_{i}) \ov{e}_{i} = \ov{\rho}_3(t_{i})$.
%One then has that 
%\begin{equation*}
%\Res_{\lat_3 \otimes_{\mc{R}} \kappa_{\mc{R}}}(\ov{E}_2 \circ \ov{\rho}_{\widetilde{\mc{R}}}(a a_3 t_{i} f_{i})) = \phi
%\end{equation*}
%where we recall that $\rho_{\widetilde{\mc{R}}}(f_{i}) = E(\alpha_{i})$.
\end{proof}

As we are viewing these morphisms of $\mc{R}$-modules as matrices with entries in $F_{\mc{R}}$, we can compute their traces.  We now rehash the work in \cite{jb08} in our current setting.

\begin{lemma} For all $a \in \mc{A}$ and $1\leq i \leq 3$ we have $\Tr(A^{i,i}(a)) \in \mc{R}$ and $$\Tr(A^{i,i}(a)) \equiv \Tr(\rho_{i}(a)) (\modu I).$$
\end{lemma}

\begin{proof}  Observe that one has $\Tr(A^{i,i}(a)) = \Tr(E_{i} \rho_{\widetilde{\mc{R}}}(a) E_{i}) = \Tr( \rho_{\widetilde{\mc{R}}}(a_{i} a a_{i})) \in \mc{R}$, and so the first statement holds.  

By assumption (2) we have that for all $a \in \mc{A}$
\begin{equation*}
\Tr(\rho_1(a_{i} a a_{i})) + \Tr(\rho_2(a_{i} a a_{i})) + \Tr(\rho_{3}(a_{i} a a_{i})) \equiv \Tr(\rho_{\widetilde{\mc{R}}}(a_{i} a a_{i})) = \Tr(A^{i,i}(a)) (\modu I).
\end{equation*}
We have that $\rho_{i}(a_{j}) \equiv 0 (\modu I)$ unless $i = j$.  Thus, we obtain 
\begin{equation*}
\Tr(A^{i,i}(a)) \equiv \Tr(\rho_{i}(a)) (\modu I)
\end{equation*}
for $1 \leq i \leq 3$.
\end{proof}

\begin{lemma}\label{lem:traces}  For all $a,b \in \mc{A}$ and $1 \leq j \leq 3$ we have  $$\Tr\left(\sum_{\substack{1 \leq i \leq 3 \\ i \neq j}} A^{j,i}(a)A^{i,j}(b)\right) \in \mc{R}$$ and 
\begin{equation*}
\Tr\left(\sum_{\substack{1 \leq i \leq 3 \\ i \neq j}} A^{j,i}(a)A^{i,j}(b)\right) \equiv \Tr\left(\sum_{\substack{1 \leq i \leq 3 \\ i \neq j} }A^{j,i}(b)A^{i,j}(a)\right) (\modu I).
\end{equation*}
\end{lemma}

\begin{proof}  We prove the case with $j = 3$ as the others are completely analogous. Let $a, b\in \mc{A}$. Observe that 
\begin{align*}
\Tr\left(\sum_{i=1}^{2} A^{3,i}(a)A^{i,3}(b)\right) &= \Tr(E_3\rho_{\widetilde{\mc{R}}}(ab) E_3) \\
    &= \Tr(\rho_{\widetilde{\mc{R}}}(a_3 ab a_3)) \in \mc{R}.
\end{align*}
This proves the first claim.

Using the previous lemma we have that 
\begin{align*}
\Tr(A^{3,3}(ab)) &\equiv \Tr(\rho_3(ab)) (\modu I) \\
        &= \Tr(\rho_3(a)\rho_3(b)) \\
        &= \Tr(\rho_3(b)\rho_3(a))\\
        &= \Tr(\rho_3(ba))\\
        &\equiv \Tr(A^{3,3}(ba)) (\modu I).
\end{align*}
Thus, $\Tr(A^{3,3}(ab)) \equiv \Tr(A^{3,3}(ba)) (\modu I)$ for all $a,b \in \mc{A}$.  We combine this with the fact that $\rho_{\widetilde{\mc{R}}}(ab) = \rho_{\widetilde{\mc{R}}}(a)\rho_{\widetilde{\mc{R}}}(b)$ to reach the desired congruence modulo $I$.
\end{proof}  

\begin{lemma}\label{lem:urbanstar}  For all $a \in \mc{A}$ we have $A^{3,i}(a) \in \Hom_{\mc{R}}(\lat_{i}, \mf{m}_{\mc{R}} \lat_3)$ for $i=1,2$.  Moreover, if there exists an $a \in \mc{A}$ so that $A^{1,2}(a) \neq 0$, then $A^{2,1} \in \Hom_{\mc{R}}(\lat_1, \mf{m}_{\mc{R}}\lat_2)$.
\end{lemma}

\begin{proof} Suppose there exists $a \in \mc{A}$ so that $A^{3,i}(a) \notin \Hom_{\mc{R}}(\lat_{i}, \mf{m}_{\mc{R}} \lat_3)$.  Then there exists $\wp$ so that after localization at $\wp$ one has $\rho_{\wp}(r)(\lat_{i, \wp})$ is not contained in $\mf{m}_{\wp}\lat_{3, \wp}$.  However, this contradicts equation (\ref{eqn:urbanstar}).  Similarly, one has from equation (\ref{eqn:urbanstar}) that $\star_1 \star_3 = 0$.  This allows us to make the conclusion in the case where $A^{1,2}$ is not exactly 0.  
\end{proof}

We use this lemma as our base case in an induction argument as used in \cite{jb08} and \cite{urbanduke}.  We prove that for $j \geq 1$ and $1 \leq i \leq 2$ we have $A^{3,i}(a) \in \Hom_{\mc{R}}(\lat_{i}, (\mf{m}_{\mc{R}}^{j} + I)\lat_3)$ for all $a \in \mc{A}$.  Moreover, if there exists $a_0 \in \mc{A}$ so that $A^{1,2}(a_0) \neq 0$, then $A^{2,1}(a) \in \Hom_{\mc{R}}(\lat_1, (\mf{m}_{\mc{R}}^{j} + I)\lat_2)$ for all $a \in \mc{A}$.  Assume inductively that the statement is true for some $j$. We break the proof into two steps.  First we prove that the statement is true for $a \in \ker(\rho \otimes_{\mc{R}} \kappa_{\mc{R}})$.

\begin{lemma}\label{lem:inductionfirststep} Let $a \in \ker(\rho \otimes_{\mc{R}} \kappa_{\mc{R}})$. Then under our induction hypothesis we have 
\begin{equation*}
A^{3,i}(a) \in \Hom_{\mc{R}}(\lat_{i}, (\mf{m}_{\mc{R}}^{j+1} + I)\lat_3)
\end{equation*}
for $1 \leq i \leq 2$. Moreover, if there exists $a_0 \in \mc{A}$ so that $A^{1,2}(a_0) \neq 0$, then $A^{2,1}(a) \in \Hom_{\mc{R}}(\lat_1, (\mf{m}_{\mc{R}}^{j} + I)\lat_2)$.
\end{lemma}

\begin{proof}  The fact that $a \in \ker(\rho \times_{\mc{R}}\kappa_{\mc{R}})$ implies that $A^{i,3}(a)(\lat_{3}) \subset \mf{m}_{\mc{R}} \lat_{i}$ for $i =1,2$.  Thus, our induction hypothesis gives that for any $b \in \mc{A}$ we have
\begin{equation*}
A^{3,i}(b)A^{i,3}(a)(\lat_{i}) \subset A^{3,i}(b)(\mf{m}_{\mc{R}}(\lat_{i}) = \mf{m}_{\mc{R}} A^{3,i}(b)(\lat_{i}) \subset (\mf{m}_{\mc{R}}^{j+1} + I)\lat_{3}
\end{equation*}
for $i=1,2$. Thus, by Lemma \ref{lem:traces} we have 
\begin{equation*}
\Tr\left( \sum_{i=1}^{2} A^{3,i}(a) A^{i,3}(b) \right) \in \mf{m}_{\mc{R}}^{j+1} + I.
\end{equation*}

We now finish the proof for $i=1$ as the case $i=2$ is completely analogous. Let $\ell \in \lat_1$.  We wish to show that $A^{3,1}(a)\ell \in (\mf{m}_{\mc{R}}^{j+1} + I)\lat_3$. Write 
\begin{equation*}
A^{3,1}(a)\ell = \sum_{i=n_2+1}^{n} \alpha_{i} e_{i}'
\end{equation*}
for some $\alpha_{i} \in \mc{R}$.   Lemma \ref{lem:surj} guarantees that for each $n_2+1 \leq i \leq n$ there exists a $t_{i} \in \mc{A}$ so that 
$A^{1,3}(t_{i})(e_{i}') = \ell$, $A^{1,3}(t_{i})(e_{j}') = 0$ for $j\neq i$ and $A^{2,3}(t_{i})(e_{j}') = 0$ for all $n_2+1 \leq j \leq n$.  Thus, we have $\alpha_{i} = \Tr(A^{3,1}(a)A^{1,3}(t_{i})) \in \mf{m}_{\mc{R}}^{j+1} + I$ for each $n_2 + 1 \leq i \leq n$, which gives the first result.

Note that all of the input into the proof of the first result remains true for the second statement under our assumption that $A^{1,2}(a_0) \neq 0$ for some $a_0 \in \mc{A}$.  Thus, the exact same proof works for the second statement as well.  
\end{proof}

\begin{lemma} Under our induction hypothesis we have $A^{3,i}(a) \in \Hom_{\mc{R}}(\lat_{i}, (\mf{m}_{\mc{R}}^{j} + I)\lat_3)$ for all $a \in \mc{A}$ and $i=1,2$.  Moreover, if there exists $a_0 \in \mc{A}$ so that $A^{1,2}(a_0) \neq 0$, then $A^{2,1}(a) \in \Hom_{\mc{R}}(\lat_1, (\mf{m}_{\mc{R}}^{j} + I)\lat_2)$ for all $a \in \mc{A}$.
\end{lemma}

\begin{proof}  Consider $\IM \rho \otimes \kappa_{\mc{R}} \subset \Hom(\lat \otimes \kappa_{\mc{R}}, \lat \otimes \kappa_{\mc{R}})$.  We denote the projection of $\ov{\rho}(a)$ onto $\Hom(\lat_{i} \otimes \kappa_{\mc{R}}, \lat_{j} \otimes \kappa_{\mc{R}})$ by $\ov{A}^{j,i}$. Applying Lemma \ref{lem:urbanstar} we have a decomposition
\begin{equation*}
\IM \rho \otimes \kappa_{\mc{R}} = \sum_{1 \leq i \leq j \leq 3} (\IM \rho \otimes \kappa_{\mc{R}})_{i,j}
\end{equation*}
and so we can denote any element of $\IM \rho \otimes \kappa_{\mc{R}}$ by a matrix
\begin{equation*}
\begin{pmatrix} \ov{A}^{1,1} & \ov{A}^{1,2} & \ov{A}^{1,3} \\ \ov{A}^{2,1} & \ov{A}^{2,2} & \ov{A}^{2,3} \\
                0 & 0 & \ov{A}^{3,3} \end{pmatrix}
\end{equation*}
with $\ov{A}^{i,j} \in (\IM \rho \otimes \kappa_{\mc{R}})_{i,j}$.

Lemma \ref{lem:inductionfirststep} gives well-defined linear maps
\begin{equation*}
\Phi_{3,j}: \IM \rho \otimes \kappa_{\mc{R}} \rightarrow \Hom_{\mc{R}}(\lat_{j}, (\mf{m}_{\mc{R}}^{j} + I)\lat_3/ (\mf{m}_{\mc{R}}^{j+1} + I)\lat_3)
\end{equation*}
for $j=1,2$ induced by the map $a \mapsto A^{3,j}(a)$.  To finish the induction it is enough to show that $\Phi_{3,1}$ and $\Phi_{3,2}$ are both 0.  Observe that by definition $\Phi_{3,1}$ and $\Phi_{3,2}$ are zero on diagonal matrices.  The relations 
\begin{equation*}
A^{3,k}(ab) = \sum_{i=1}^{3} A^{3,i}(a)A^{i,k}(b)
\end{equation*}
gives the equations
\begin{equation*}
\Phi_{3,k}(\ov{B}\, \ov{C}) = \Phi_{3,1}(\ov{B}) \ov{C}^{1,k} + \Phi_{3,2}(\ov{B}) \ov{C}^{2,k} + \ov{B}^{3,3} \Phi_{3,k}(\ov{C})
\end{equation*}
for $j=1,2$.  It is enough to show that for each $1 \leq i_0 \leq j_0 \leq 3$,  $\Phi_{3,k}(\ov{B}) = 0$ for $\ov{B}$ defined by $\ov{B}^{i,j} = 0$ unless $i = i_0, j = j_0$.  We have
\begin{align*}
\Phi_{3,k}(\ov{B}) &= \Phi_{3,k}(\ov{B} \ov{\rho}(a_{j_0})) \\
    &= \Phi_{3,1}(\ov{B}) \ov{\rho}(a_{j_0})^{1,1} + \Phi_{3,2}(\ov{B}) \ov{\rho}(a_{j_0})^{2,j} + \ov{B}^{3,3} \Phi_{3,j}(\ov{\rho}(a_{j_0}))\\
    &=0
\end{align*}
where we have used that $\Phi_{3,j}(\ov{B}) = 0$ for $j=1,2$ and $\Phi_{3,j}(\ov{\rho}(a_{j_0})) = 0$ for $j=1,2$ and we recall that the $a_{i}$ were defined earlier such that $\rho(a_{i}) = E_{i}$. This completes the proof of the first statement and hence the induction.  The exact same argument gives the second statement as well.
\end{proof}

We summarize what we have proven thus far in the following proposition.  The only point that needs mentioning is the statement about the action on the quotient also requires Theorem 1 in \cite{carayol}.

\begin{proposition}\label{prop:lattices1} The lattice $(\lat_1 \otimes \mc{R}/I) \oplus (\lat_2 \otimes \mc{R}/I)$ is stable under the action of $\mc{A}$ and the action of $\mc{A}$ on the quotient $(\lat \otimes \mc{R}/I)/((\lat_1 \otimes \mc{R}/I) \oplus (\lat_2 \otimes \mc{R}/I))$ is isomorphic to $\rho_{3} \otimes \mc{R}/I$.  Moreover, either $\Hom_{\mc{R}/I}(\lat_1 \otimes \mc{R}/I, \lat_2 \otimes \mc{R}/I) = 0$ or $\Hom_{\mc{R}/I}(\lat_2 \otimes \mc{R}/I, \lat_1 \otimes \mc{R}/I) = 0$. 
\end{proposition}

It now remains to consider the action of $\mc{A}$ on $(\lat_1 \otimes \mc{R}/I) \oplus (\lat_2 \otimes \mc{R}/I)$.  

\begin{lemma}\label{lem:lattices2}
\begin{enumerate}
\item The $\mc{R}$-modules $$\lat_1(\alpha_{i}) = E(\alpha_{i})\lat_1 = \ker(\rho_{\widetilde{\mc{R}}}(a_0) - \alpha_i \Id)\cap \lat_1$$ are mutually isomorphic to some module $\mathscr{T}_1$ for $1 \leq i \leq n_2$ and similarly the $\mc{R}$-modules $$\lat_2(\alpha_{i}) = E(\alpha_{i})\lat_2 = \ker(\rho_{\widetilde{\mc{R}}}(a_0) - \alpha_i \Id)\cap \lat_2$$ are mutually isomorphic to some module $\mathscr{T}_2$ for $n_1+1 \leq i \leq n_2$.  \\
\item As an $\mc{A}$-module, $\lat_{j} \otimes \mc{R}/I \cong \rho_{j} \otimes \mc{T}_{j}/I\mc{T}_{j}$ for $j=1,2$.
\end{enumerate}
\end{lemma}

\begin{proof}  We prove this theorem for $\lat_1$ as the argument for $\lat_2$ is exactly the same.  Recall that we let $f_{i} \in \mc{A}$ denote the element so that $\rho_{\widetilde{\mc{R}}}(f_{i}) = \diag(0,\dots, 0, 1, 0, \dots, 0)$ with a 1 in the $i$th place.  Given any $i,j \in \{ 1, \dots, n_1\}$, the irreducibility of $\ov{\rho}_1$ gives an element $\sigma_{i,j} \in \mc{A}$ so that $\rho_{\widetilde{\mc{R}}}(\sigma_{i,j}) \ov{e}_{i} = \ov{e}_{j}$.  Thus, $\rho_{\widetilde{\mc{R}}}(f_{j} \sigma_{i,j} f_{i})$ gives a morphism $\phi_{i,j}$ from $\lat_1(\alpha_{i})$ to $\lat_{1}(\alpha_{j})$. Our choice of $f_{i}$'s and $\sigma_{i,j}$ give that $\phi_{i,j} \circ \phi_{j,i}$ gives an automorphism of $\lat_{1}(\alpha_{i})$ and so $\phi_{i,j}$ is an isomorphism.  This gives the first result where we set $\mc{T}_1 = \lat_1(\alpha_1)$ and $\mc{T}_2 = \lat_2(\alpha_{n_1+1})$.  

Fix an isomorphism $\lat_1 \otimes \mc{R}/I \cong (\mc{T}_1/I \mc{T}_1)^{n_1}$.  Set $\mc{E} = \End_{\mc{R}/I}(\mc{T}_1/I \mc{T}_1)$ (a noncommutative Artinian algebra) and $\theta: \mc{R}/I \rightarrow \mc{E}$ the canonical algebra homomorphism.  The action of $\mc{A}$ on $\lat_1 \otimes \mc{R}/I$ gives a representation $\rho_1' in \M_{n_1}(\mc{E})$ so that $\Tr(\rho_1'(a)) \in \theta(\mc{R}/I)$ is defined for all $a \in \mc{A}$ and we have $\Tr(\rho_1'(a)) = \theta(\Tr(\rho_1(a)))$ for all $a \in \mc{A}$.  A generalization of Theorem 1.1.2 in \cite{carayol} then gives the result.  See the proof of Lemma 1.5 in \cite{urbanduke} for the proof of this generalization.
\end{proof}

\begin{proof} (of Theorem \ref{thm:lattices}) In light of Proposition \ref{prop:lattices1} and Lemma \ref{lem:lattices2}, it only remains to prove the last statement of the theorem.  Suppose some quotient of $\lat$ is isomorphic to $\ov{\rho}'$ with $\ov{\rho}'^{\semi} = \ov{\rho}_1 \oplus \ov{\rho}_2$.  Let $\lat'$ be the sublattice of $\lat$ that is stable under the action of $\mc{A}$ so that $\lat/\lat' \cong \ov{\rho}'$.  From our decomposition of $\lat$ we inherit a decomposition of $\lat'$ as $\lat' = \lat_1 ' \oplus \lat_2 ' \oplus \lat_3 '$.  Thus, we have 
\begin{equation*}
\lat/\lat' = \lat_1 /\lat_1 ' \oplus \lat_2/\lat_2 ' \oplus \lat_3/\lat_3 ' \cong \ov{\rho}'.
\end{equation*}
However, our assumption on $\ov{\rho}'^{\semi}$ gives that $\lat_3 = \lat_3 '$.  This, combined with the fact that $\lat$ is generated by $\lat_3$ over $\mc{A}$ gives $\lat = \lat'$, a contradiction.  Thus, no such quotient can exist.  It is useful to observe that what this is saying in terms of the matrices is that $A^{1,3}$ and $A^{2,3}$ cannot both be 0.
\end{proof}

\section{Selmer groups}\label{sec:selmer}
Let $K$ be a field and write $\gal{K}$ for $\Gal(\ov{K}/K)$. Let $M$ be a topological $\gal{K}$-module.  We write the cohomology group $\coh^1_{\cont}(\gal{K},M)$ as $\coh^1(K,M)$ where ``cont'' refers to continuous cocycles.  For a prime $\ell$, we write $D_{\ell}$ for the decomposition group at $\ell$ and identify it with $\gal{\rat_{\ell}}$. 

Let $E/\rat_{p}$ be a finite extension.  Let $\ROI$ be the ring of integers of $E$ and $\varpi$ a uniformizer.  Let $V$ be a finite dimensional Galois representation over $E$.  We will also find it convenient to write $\rho: \absgal \rightarrow \GL_{n}(E)$ to denote the Galois representation $V$ when $\dim_{E}(V)=n$.  We switch interchangably between these notations depending upon context.  Let $T \subseteq V$ be a Galois-stable $\ROI$-lattice, i.e., $T$ is stable under the action of $\absgal$ and $T \otimes_{\ROI} E \cong V$. Set $W = V/T$.

We write $\mathbb{B}_{\cris}$ for the ring of $p$-adic periods (\cite{fontaine}). Set
\begin{equation*}
D = (V \otimes_{\rat_{p}} \mathbb{B}_{\cris})^{D_{p}}
\end{equation*}
and
\begin{equation*}
\Cris(V) = \coh^0(\rat_{p}, V \otimes_{\rat_{p}} \mathbb{B}_{\cris}).
\end{equation*}
We say the representation $V$ is {\it crystalline} if $\dim_{\rat_{p}} V = \dim_{\rat_{p}} \Cris(V)$.  Let $\Fil^{i} D$ be a decreasing filtration of $D$.  If $V$ is crystalline, we say $V$ is {\it short} if $\Fil^0 D = D$, $\Fil^{p} D = 0$, and if whenever $V'$ is a nonzero quotient of $V$, then $V' \otimes_{\rat_{p}} \rat_{p}(p-1)$ is ramified.  Note that $\rat_{p}(n)$ is the 1-dimensional space over $\rat_{p}$ on which $\absgal$ acts via the $n$th power of the $p$-adic cyclotomic character. 

The local Selmer groups are defined as follows.  Set
\begin{equation*}
\coh_{f}^1(\rat_{\ell}, V) = \left\{\begin{array}{ll} \coh^1_{\ur}(\rat_{\ell}, V) & \ell \neq p \\ \ker(\coh^1(\rat_{p},V) \rightarrow \coh^1(\rat_{p}, V \otimes_{\rat_{p}}\mathbb{B}_{\cris})) & \ell =p \end{array} \right.
\end{equation*}
where 
\begin{equation*}
\coh^1_{\ur}(\rat_{\ell}, M) = \ker(\coh^1(\rat_{\ell},M) \rightarrow \coh^1(I_{\ell}, M))
\end{equation*}
for any $D_{\ell}$-module $M$ where $I_{\ell}$ is the inertia group at $\ell$.  With $W$ as above we define
$\coh_{f}^1(\rat_{\ell},W)$ to be the image of $\coh^1_{f}(\rat_{\ell},V)$ under the natural map $\coh^1(\rat_{\ell}, V) \rightarrow \coh^1(\rat_{\ell}, W)$. 

\begin{definition} The Selmer group of $W$ is given by
\begin{equation*}
\Sel(W)  = \ker \left( \coh^1(\rat, W) \rightarrow \bigoplus_{\ell} \frac{\coh^1(\rat_{\ell}, W)}{\coh_{f}^1(\rat_{\ell}, W)}\right),
\end{equation*}
i.e., it is the cocycles $\mathfrak{c} \in \coh^1(\rat,W)$ that lie in $\coh^1_{f}(\rat_{\ell},W)$ when restricted to $D_{\ell}$.
\end{definition}

Before we can define the degree $n$ Selmer groups of interest we must recall the notion of extensions of modules and the relationship between these extensions and the first cohomology group. An extension of $M$ by $N$ is a short exact sequence 

\begin{figure}[h!]
\centerline{\xymatrix{ 0 \ar[r] & N \ar[r]^{\alpha} & X \ar[r]^{\beta} & M \ar[r] & 0}}
\end{figure}

\noindent where $X$ is a $R[G]$-module and $\alpha$ and $\beta$ are $R[G]$-homomorphisms.  We sometimes refer to such an extension as the extension $X$. We say two extensions $X$ and $Y$ are equivalent if there is a $R[G]$-isomorphism $\gamma$ making the following diagram commute\\

\begin{figure}[h!]
\centerline{\xymatrix{ 0 \ar[r] & N \ar[r]^{\alpha_{X}} \ar[d]_{\id_{N}}& X \ar[r]^{\beta_{X}} \ar[d]_{\gamma} & M \ar[r] \ar[d]_{\id_{M}}& 0 \\
            0 \ar[r] & N \ar[r]^{\alpha_{Y}}& Y \ar[r]^{\beta_{Y}} & M \ar[r]& 0.}}
\end{figure}

\noindent Let $\Ext^1_{R[G]}(M,N)$ denote the set of equivalence classes of $R[G]$-extensions of $M$ by $N$ which split as extensions of $R$-modules, i.e., if $X$ is the extension of $M$ by $N$, then $X \cong M \oplus N$ as $R$-modules. 

The following result will allow us to appropriately define the degree $n$ Selmer group.  The case where $M=N$ is given as Proposition 4 in \cite{mfflt}.  

\begin{theorem} (Theorem 9.2, \cite{jbsp4}) Let $M$ and $N$ be $R[G]$-modules.  There is a one-one correspondence between the sets $\coh^1(G, \Hom_{R}(M,N))$ and $\Ext^1_{R[G]}(M,N)$.
\end{theorem} 

The map from $\Ext^1_{R[G]}(M,N)$ to $\coh^1(G, \Hom_{R}(M,N))$ is given as follows.  Let\\

\begin{figure}[h!]
\centerline{\xymatrix{ 0 \ar[r] & N \ar[r]^{\alpha} & X \ar[r]^{\beta} & M \ar[r] \ar@/_1pc/[l]_{s_{X}}& 0}}
\end{figure}

\noindent be an extension with $s_{X}$ a $R$-section of $X$.  This extension is mapped to the cohomology class $g \mapsto \mf{c}_{g}$ where $\mf{c}_{g}: M \rightarrow N$ is defined by
\begin{equation*}
\mf{c}_{g}(m) = \alpha^{-1}(\rho(g)s_{X}(\rho_{M}(g^{-1})m) - s_{X}(m))
\end{equation*}
where $\rho$ denotes the $G$-action on $X$ and $\rho_{M}$ the $G$-action on $M$.

As it will be useful later, we briefly consider the case where $N= N_1 \oplus N_2$. In this case we have that $$\Ext_{R[G]}^1(M,N) \cong \prod_{i=1}^{2} \Ext(M,N_{i})$$ and 
\begin{align*}
\coh^1(G, \Hom_{R}(M,N)) &\cong \coh^1(G, \Hom_{R}(M,N_1) \oplus \Hom_{R}(M,N_2)) \\
    &\cong \coh^1(G, \Hom_{R}(M,N_1)) \oplus \coh^1(G, \Hom_{R}(M,N_2)).
\end{align*}
Thus, in this case given an extension $X \in \Ext^1_{R[G]}(M,N)$, we obtain cohomology classes $\mf{c}_{1} \in \coh^1(G, \Hom_{R}(M,N_1))$ and $\mf{c}_{2} \in \coh^1(G, \Hom_{R}(M, N_2))$.

Let $W[n]$ be the $\ROI$-submodule of $W$ consisting of elements killed by $\varpi^{n}$.  The previous theorem gives a bijection between $\Ext_{(\ROI/\varpi^{n})[D_{p}]}^1(\ROI/\varpi^{n}, W[n])$ and $\coh^1(D_{p}, W[n])$.  For $\ell \neq p$, we define the local degree n Selmer groups by $\coh^1_{f}(\rat_{\ell}, W[n]) = \coh_{\ur}^1(\rat, W[n])$.  At the prime $p$ we define the local degree n Selmer group to be the subset of classes of extensions of $D_{p}$-modules\\

\begin{figure}[h!]
\centerline{\xymatrix{0 \ar[r] & W[n] \ar[r]&X \ar[r]& \ROI/\varpi^{n} \ar[r]&0}}
\end{figure}

\noindent where $X$ lies in the essential image of the functor $\mathbb{V}$ defined in $\S$ 1.1 of \cite{dfg}.  The precise definition of $\mathbb{V}$ is technical and is not needed here.  We content ourselves with stating that this essential image is stable under direct sums, subobjects, and quotients (\cite{dfg}, $\S$ 2.1).  For our purposes the following two propositions are what is needed.

\begin{proposition}\label{prop:crystalline} (\cite{dfg}, p. 670) If $V$ is a short crystalline representation at $p$, $T$ a $D_{p}$-stable lattice, and $X$ a subquotient of $T/\varpi^{n}$ that gives an extension of $D_{p}$-modules as above, then the class of this extension is in $\coh_{f}^1(\rat_{p}, W[n])$. 
\end{proposition}

\begin{proposition}(\cite{dfg}, Proposition 2.2)\label{thm:diamondflachguo} The
natural isomorphism
\begin{equation*}
\varinjlim_{n} \coh^1(\rat_{\ell}, W[n]) \cong
\coh^1(\rat_{\ell}, W)
\end{equation*}
induces isomorphisms
\begin{equation*}
\varinjlim_{n} \coh_{\ur}^1(\rat_{\ell}, W[n]) \cong
\coh_{\ur}^1(\rat_{\ell},W)
\end{equation*}
and
\begin{equation*}
\varinjlim_{n} \coh_{f}^1(\rat_{p}, W[n]) \cong
\coh_{f}^1(\rat_{p}, W).
\end{equation*}
\end{proposition}

Let $M$ be a $\inte_{p}$-module. We denote the Pontryagin dual $\Hom_{\cont}(M, \rat_{p}/\inte_{p})$ of $M$ by $M^{\vee}$.  In particular, we denote the dual of the Selmer group $\Sel(W)$ by $S(W)$ to ease the notation.

We close this section with the following results on $S(W)$. 

\begin{lemma} (\cite{klosin} Lemma 9.4) $S(W)$ is a finitely generated $\ROI$-module.
\end{lemma}

\begin{lemma} (\cite{klosin} Lemma 9.5) If the modulo $\varpi$ reduction $\ov{\rho}$ of $\rho$ is absolutely irreducible, then the length of $S(W)$ as an $\ROI$-module is independent of the choice of the lattice $T$.
\end{lemma}

\section{A lower bound on the Selmer group}\label{sec:finalresult} 

Let $E/\rat_{p}$ be a finite extension as before large enough so that our results from $\S$ \ref{sec:congruences} are defined over $E$.  We enlarge $E$ when necessary so that the appropriate Galois representations in this section are defined over $E$ as well.  Let $\ROI$ be the ring of integers of $E$, $\varpi$ the uniformizer, $\p = (\varpi)$ the prime ideal over $p$, and $\F$ the residue field.  

Let $\rho_{f}: \absgal \rightarrow \GL(V_{f, \p})$ be the $p$-adic Galois representation associated to an eigenform $f$, $T_{f,\p}$ a $\absgal$-stable $\ROI$-lattice, and $W_{f,\p} = V_{f,\p}/T_{f,\p}$.  We denote twists by the $m$th power of the cyclotomic character by writing $V_{f,\p}(m)$ and similarly for $W_{f,\p}(m)$.  We drop the subscript $f$ and $\p$ except for in the statement of theorems as they are fixed throughout the section, i.e., we set $W = W_{f,\p}, T =T_{f, \p}$ and $V = V_{f,\p}$.  

We have the following result giving the existence of $4$-dimensional Galois representations attached to Siegel eigenforms.   

\begin{theorem}(\cite{skinnerurban}, Theorem 3.1.3)\label{thm:existencegalreps} Let $F \in \mathcal{S}_{k}(\Sp_4(\inte))$ be an eigenform, $K_{F}$ the number field generated by the Hecke eigenvalues of $F$, and $\wp$ a prime of $K_{F}$ over $p$.  There exists a finite extension $E$ of the completion of $K_{F, \wp}$ of $K_{F}$ at $\wp$ and a continuous semi-simple Galois representation 
\begin{equation*}
\rho_{F, \wp}: \absgal \rightarrow \GL_4(E)
\end{equation*}
unramified away from $p$ so that for all $\ell \neq p$ we have
\begin{equation*}
\det(X\cdot 1_{4} - \rho_{F, \wp}(\Frob_{\ell})) = L_{\spin, (\ell)}(X).
\end{equation*}
\end{theorem}

The following result is crucial in producing elements in the Selmer group.

\begin{theorem} (\cite{faltings}, \cite{urban}) Let $F$ be as in Theorem \ref{thm:existencegalreps}.  The restriction of $\rho_{F, \wp}$ to the decomposition group $D_{p}$ is crystalline at $p$.  In addition, if $p > 2k-2$ then $\rho_{F, \wp}$ is short.
\end{theorem}

Recall that we denoted the image of $\T_{\ROI}^{S}$ in $\End_{\cmplx}(\mathcal{S}_{k}^{\NM}(\Sp_4(\inte)))$ by $\T_{\ROI}^{\NM}$.  Let $\mathcal{M}^{S}$ denote the set of maximal ideals of $\T_{\ROI}^{S}$ and $\mathcal{M}^{\NM}$ denote the set of maximal ideals of $\T_{\ROI}^{\NM}$. Write $\T_{\ROI}^{\NM} = \prod_{\mf{m}\in \mathcal{M}^{\NM}} \T_{\ROI, \mf{m}}^{\NM}$ where the subscript $\mf{m}$ denotes localization at $\mf{m}$.  Again we let $\phi$ denote the natural projection from $\T_{\ROI}^{S}$ to $\T_{\ROI}^{\NM}$.  Let $\mathcal{M}^{c}$ denote the set of primes in $\mathcal{M}^{S}$ that are preimages of elements of $\mathcal{M}^{\NM}$ under $\phi$ and $\mathcal{M}^{nc} = \mathcal{M}^{S} - \mathcal{M}^{c}$. Our factorization of $\T_{\ROI}^{S}$ and $\T_{\ROI}^{\NM}$ allows us to factor the map $\phi$ as 
\begin{equation*}
\phi = \prod_{\mf{m} \in \mathcal{M}^{c}} \phi_{\mf{m}} \prod \times \prod_{\mf{m} \in \mathcal{M}^{nc}} 0_{\mf{m}}
\end{equation*}
where $0_{\mf{m}}$ is the zero map and $\phi_{\mf{m}}: \T_{\ROI, \mf{m}}^{S} \rightarrow \T_{\ROI, \mf{m}'}^{S}$ is the projection with $\mathfrak{m}' \in \mathcal{M}^{\NM}$ the unique maximal ideal so that $\phi^{-1}(\mf{m}') = \mf{m}$. 

\begin{theorem}\label{thm:selmerbound} We have
\begin{equation*}
\ord_{p}(\# S(\rat, W_{f,\p}(1-k))) \geq \ord_{p}(\# \T_{\mf{m}_{\Ff}}^{\NM}/\phi_{\mf{m}_{\Ff}}(\Ann(\Ff)))
\end{equation*}
where for an $\ROI$-module $M$, $\ord_{p}(\#M) = [\ROI/\varpi : \F_{p}] \length_{\ROI}(M)$.
\end{theorem}

\begin{corollary} Let $k > 3$ be an integer and $p$ a prime so that $p > 2k-2$.  Let $f \in S_{2k-2}(\SL_2(\inte),\ROI)$  be a newform with real Fourier coefficients and $\Ff$ the Saito-Kurokawa lift of $f$.  Let $f$ be ordinary at $p$ and $\overline{\rho}_{f}$ be irreducible.  If there exists $N>1$, a fundamental discriminant $D$ so that $(-1)^{k-1}D>0$, $\chi_{D}(-1)=-1$, $p \nmid ND[\Sp_4(\inte):\Gamma_0^4(N)]$, and a Dirichlet character $\chi$ of conductor $N$ so that 
\begin{equation*}
-M=\ord_{\varpi}(\mathcal{L}(k,f,D,\chi)) <0
\end{equation*}
then 
\begin{equation*}
 \ord_{p}(\# S(\rat, W_{f,\p}(1-k))) \geq M
\end{equation*}
where we recall
\begin{equation*}
\mathcal{L}(k,f,D, \chi) = \frac{L^{\Sigma}(3-k,\chi) L_{\alg}(k-1,f,\chi_{D})L_{\alg}(1,f,\chi) L_{\alg}(2,f,\chi)}{L_{\alg}(k,f)}.
\end{equation*}
\end{corollary}

\begin{proof} This corollary is an immediate consequence of Theorem \ref{thm:selmerbound} and Corollary \ref{corl:capideal}.
\end{proof}

The work in $\S$ \ref{sec:lattices} and in particular Theorem \ref{thm:lattices} is the main input into the proof of Theorem \ref{thm:selmerbound}. We begin by adapting Theorem \ref{thm:lattices} to our current situation.  First we set up some notation following \cite{klosin}.  

Let $n = n_1 + n_2 + n_3$ with $n_{i} \geq 1$.  Let $V_{i}$ be $E$ vector spaces of dimension $n_{i}$ affording continuous absolutely irreducible representations $\rho_{i}: \absgal \rightarrow \Aut_{E}(V_{i})$ for $1 \leq i \leq 3$.  Assume the residual representations $\ov{\rho}_{i}$ are irreducible and nonisomorphic for $1 \leq i \leq 3$.   Let $\mc{V}_{1}, \dots, \mc{V}_{m}$ be $n$-dimensional $E$ vector spaces affording absolutely irreducible continuous representations $\varrho_{i} : \absgal \rightarrow \Aut_{E}(\mc{V}_{i})$ for $1 \leq i \leq m$.  Further assume that the modulo $\varpi$ reductions of $\varrho_{i}$ satisfy 
\begin{equation*}
\ov{\varrho}_{i}^{\semi} \cong \ov{\rho}_1 \oplus \ov{\rho}_2 \oplus \ov{\rho}_3
\end{equation*}
for some $\absgal$-stable lattice in $\mc{V}_{i}$ (and hence for all such lattices.)

For each $\sigma \in \absgal$, let 
\begin{equation*}
\sum_{j=0}^{n} a_{j}(\sigma)X^{j} \in \ROI[X]
\end{equation*}
be the characteristic polynomial of $(\rho_1 \oplus \rho_2 \oplus \rho_3)(\sigma)$ and 
\begin{equation*}
\sum_{j=0}^{n} c_{j}(i,\sigma)X^{j} \in \ROI[X]
\end{equation*}
be the characteristic polynomial of $\varrho_{i}(\sigma)$.  Set 
\begin{equation*}
c_{j}(\sigma) = \begin{pmatrix} c_{j}(1,\sigma) \\ \vdots \\ c_{j}(m,\sigma) \end{pmatrix} \in \ROI^{m}
\end{equation*}
for $0 \leq j \leq n-1$.  Let $\T \subset \ROI^{m}$ be the $\ROI$-subalgebra generated by the set $\{ c_{j}(\sigma) : \sigma \in \absgal, 0 \leq j \leq n-1\}$.  We can use the continuity of the $\varrho_{i}$ along with the Chebotarev density theorem to conclude that 
\begin{equation*}
\T = \{ c_{j}(\Frob_{\ell}): 0 \leq j \leq n-1, \ell \neq p\}.
\end{equation*}
Observe that $\T$ is a finite $\ROI$-algebra.  Let $I \subset \T$ be the ideal generated by the set $\{ c_{j}(\Frob_{\ell} - a_{j}(\Frob_{\ell}) : 0 \leq j \leq n-1, \ell \neq p\}.$  The definition of $I$ gives that the map $\ROI \rightarrow \T/I$ giving the $\ROI$-algebra structure is surjective.  Let $J$ be the kernel of this map so that we have $\ROI/J \cong \T/I$.

\begin{corollary}\label{corl:lattices} Suppose $\F^{\times}$ contains at least $n$ distinct elements. Then there exists a $\absgal$-stable $\T$-submodule $\lat \subset \bigoplus_{i=1}^{m} \mc{V}_{i}$, $\T$-submodules $\lat_1, \lat_2$ and $\lat_3$ contained in $\lat$ and finitely generated $\T$-modules $\mc{T}_1, \mc{T}_2$ such that 
\begin{enumerate}
\item as $\T$-modules we have $\lat = \lat_1 \oplus \lat_2 \oplus \lat_3$ and $\lat_{i} \cong \T^{n_{i}}$ for $1 \leq i \leq 3$;
\item $\lat$ has no $\T[\absgal]$-quotient isomorphic to $\ov{\rho}'$ where $\ov{\rho}'^{\semi} = \ov{\rho}_1 \oplus \ov{\rho}_2$;
\item $(\lat_1 \oplus \lat_2)/I(\lat_1 \oplus \lat_2)$ is $\absgal$-stable and there exists a $\T[\absgal]$-isomorphism
\begin{equation*}
\lat/(\lat + I(\lat_1 \oplus \lat_2)) \cong M_3 \otimes_{\ROI} \T/I
\end{equation*}
for any $\absgal$-stable $\ROI$-lattice $M_3 \subset V_3$;
\item one has either $\Hom_{\T/I}(M_1 \otimes_{\ROI} \mc{T}_1/I\mc{T}_1, M_2 \otimes_{\ROI} \mc{T}_2/I\mc{T}_2)  = 0$ or $\Hom_{\T/I}(M_2 \otimes_{\ROI} \mc{T}_2/I\mc{T}_2, M_1 \otimes_{\ROI} \mc{T}_1/I\mc{T}_1) = 0$ for any $\absgal$-stable $\ROI$-lattices $M_{i} \subset V_{i}$ for $i=1,2$ ;
\item $\Fitt_{\T}(\mc{T}_{i}) = 0$ for $i=1,2$ and there exists a $\T[\absgal]$-isomorphism
\begin{equation*}
\lat_{i}/I\lat_{i} \cong M_{i} \otimes_{\ROI} \mc{T}_{i}/I\mc{T}_{i}
\end{equation*}
for any $\absgal$-stable $\ROI$-lattice $M_{i} \subset V_{i}$ for $i=1,2$.
\end{enumerate}
\end{corollary}

\begin{proof}  Everything in this corollary follows immediately from Theorem \ref{thm:lattices}.  Though Fitting ideals are not mentioned there, the proof that $\Fitt_{\T}(\mc{T}_{i}) = 0$ follows immediately from our work in $\S$ \ref{sec:lattices}.  See Lemma 9.13 of \cite{klosin} for the details.
\end{proof} 

We now specialize to our situation. 
\begin{itemize}
\item $n_1 = n_3 = 1$, $n_2 = 2$; 
\item $\rho_1 = \varepsilon^{-1}$, $\rho_2 = \rho_{f}\otimes \varepsilon^{1-k}$, $\rho_3 = \id$.  Note what we are doing here is looking at the components of the $\rho_{\Ff}\otimes \varepsilon^{1-k}$.
\item $\T = \T_{\mf{m}_{\Ff}}^{\NM}$;
\item $G_1, \dots, G_{m}$ for the elements in an orthogonal eigenbasis of $\mathcal{S}_{k}^{\NM}(\Sp_4(\inte))$ such that $\phi^{-1}(\mf{m}_{G_{i}}^{\NM}) = \mf{m}_{\Ff}$. 
\item $I = $ the ideal of $\T$ generated by $\phi_{\mf{m}_{\Ff}}(\Ann \Ff)$;
\item $(\mc{V}_{i}, \varrho_{i}) = $ the representation $\rho_{G_{i}}$ for $1 \leq i \leq m$.
\end{itemize}

Let $M_{i}$ be a $\absgal$-stable $\ROI$-lattice inside $V_{i}$ for $1 \leq i \leq 3$.  We will continue to use the matrix notation as was used in $\S$ \ref{sec:lattices} when it is convenient for our purposes.  We now break into two cases depending on whether $\Hom_{\T/I}(M_1 \otimes_{\ROI} \mc{T}_1/I\mc{T}_1, M_2 \otimes_{\ROI} \mc{T}_2/I\mc{T}_2)  = 0$ or $\Hom_{\T/I}(M_2 \otimes_{\ROI} \mc{T}_2/I\mc{T}_2, M_1 \otimes_{\ROI} \mc{T}_1/I\mc{T}_1) = 0$. 

We begin with the case where $\Hom_{\T/I}(M_1 \otimes_{\ROI} \mc{T}_1/I\mc{T}_1, M_2 \otimes_{\ROI} \mc{T}_2/I\mc{T}_2)  = 0$.  We begin with the exact sequence\\

\begin{figure}[h!]
\centerline{\xymatrix{ 0 \ar[r] & \mc{N}_1 \oplus \mc{N}_2 \ar[r] & \lat \otimes \T/I \ar[r]  &\rho_3 \otimes \T/I \ar@/_1pc/[l]^{s} \ar[r] & 0,
}}
\end{figure}

\noindent where we have $\mc{N}_{i} = M_{i} \otimes_{\ROI} \mc{T}_{i}/I \mc{T}_{i}$ for $i = 1, 2$ (see Theorem \ref{thm:lattices} or Theorem \ref{corl:lattices} (3) and (5).)  As we saw in $\S$ \ref{sec:selmer}, this gives rise to a cocycle 
\begin{equation*}
\mf{c}_2 \in \coh^1(\absgal, \Hom_{\T/I}(M_3 \otimes_{\ROI} \T/I, M_2 \otimes_{\ROI} \mc{T}_2/I \mc{T}_2).
\end{equation*}
Observe that we have 
\begin{equation*}
\Hom_{\T/I}(M_3 \otimes_{\ROI} \T/I, M_2 \otimes_{\ROI} \mc{T}_2/I \mc{T}_2) \cong \Hom_{\ROI}(M_3, M_2) \otimes_{\ROI} \mc{T}_2/I\mc{T}_2
\end{equation*}
and so $\mf{c}_2$ can be regarded as a cocycle in $\coh^1(\absgal, \Hom_{\ROI}(M_3, M_2) \otimes_{\ROI} \mc{T}_2/I\mc{T}_2)$. Define a map 
\begin{align*}
\iota_2 : \Hom_{\ROI}(\mc{T}_2/I\mc{T}_2, E/\ROI) &\rightarrow \coh^1(\absgal, \Hom_{\ROI}(M_3, M_2) \otimes_{\ROI} E/\ROI)\\
    f &\mapsto (1\otimes f)(\mf{c}_2).
\end{align*}
Our assumption that $\Hom_{\T/I}(M_1 \otimes_{\ROI} \mc{T}_1/I\mc{T}_1, M_2 \otimes_{\ROI} \mc{T}_2/I\mc{T}_2)  = 0$ and the fact the that modulo $\varpi$ reduction of $\rho_{f} \otimes \varepsilon^{1-k}$ is absolutely irreducible give that we can choose $T = \Hom_{\ROI}(M_3, M_2)$ and we have $W = \Hom_{\ROI}(M_3, M_2) \otimes_{\ROI} E/\ROI$.  

\begin{lemma}\label{lem:917} $\IM(\iota_2) \subseteq \Sel(W)$.
\end{lemma}

\begin{proof}  As our representations are unramified away from $p$, the only thing remaining to prove is that the condition at $p$ is satisfied.  Observe that since $\mc{T}_2/I\mc{T}_2$ is a finitely generated $\T$-module and $\T/I \cong \ROI/J$, it is also a finitely generated $\ROI$-module. Thus, there exists a positive integer $n$ so that $\Hom_{\ROI}(\mc{T}_2/I\mc{T}_2, E/\ROI) = \Hom_{\ROI}(\mc{T}_2/I\mc{T}_2, (E/\ROI)[n])$.  Thus, we have 
\begin{equation*}
\IM(\iota_2) \subseteq \coh^1(\absgal, \Hom_{\ROI}(M_3, M_2) \otimes_{\ROI} (E/\ROI)[n]) = \coh^1(\absgal, W[n]).
\end{equation*}
Proposition \ref{thm:diamondflachguo} gives that 
\begin{equation*}
\varinjlim_{n} \coh_{f}^1(\rat_{p}, W[n]) \cong
\coh_{f}^1(\rat_{p}, W).
\end{equation*}
Thus, it is enough to show that $\IM(\iota_2) \subseteq \coh^1_{f}(\rat_{p}, W[n])$.  However, this follows from the fact that each $(\rho_{i}, \mc{V}_{i})$ is short and crystalline at $p$.
\end{proof}

\begin{lemma}\label{lem:918} $(\ker\iota_2)^{\vee} = 0$.
\end{lemma} 

\begin{proof}  Let $f \in \ker \iota_2$, $\mc{B}_{f} = (\mc{T}_2/I\mc{T}_2)/\ker f$, and $\mc{I}_{f} = (E/\ROI)/\IM f$.  Consider the short exact sequence
\begin{equation*}
0 \rightarrow \mc{B}_{f} \overset{f}{\rightarrow}  E/\ROI \rightarrow \mc{I}_{f} \rightarrow 0.
\end{equation*}
We tensor this sequence with $T$ and consider the long exact sequence of cohomology that results as well as the natural map $\phi$:\\

\begin{figure}[h!]
\centerline{\xymatrix{& \coh^1(\absgal, T \otimes_{\ROI} \mc{T}_2/I\mc{T}_2) \ar^{\phi}[d] \ar^{\coh^1(1\otimes f)}[dr] & \\ 
    \coh^0(\absgal, T \otimes_{\ROI} \mc{I}_{f}) \ar[r] & \coh^1(\absgal, T \otimes_{\ROI}\mc{B}_{f}) \ar^{\coh^1(1 \otimes f)}[r] & \coh^1(\absgal, T \otimes_{\ROI} E/\ROI).}}
\end{figure} 

\noindent  The fact that $\absgal$ acts on $M_3$ and $M_2$ in such a way as to give rise to irreducible non-isomorphic representations gives that $\coh^0(\absgal, T \otimes_{\ROI} \mc{I}_{f}) = 0$. Thus, we must have $\coh^1(1\otimes f)$ is an injective map.  Since $f \in \ker \iota_2$ by assumption, we have $\coh^1(1 \otimes f) \circ \phi(\mf{c}_2) = 0$.  Thus, the fact that $\coh^1(1\otimes f)$ is injective shows that $\mf{c}_2$ maps to 0 under the map $\phi$.  

Given any $g \in \Hom_{\ROI}(\mc{T}_2/I\mc{T}_2, E/\ROI)$, one has that $\ker g$ has finite index in $\mc{T}_2/I\mc{T}_2$.  So in particular, we have that there exists an $\ROI$-module $A$ with $\ker f \subseteq A \subset \mc{T}_2/I\mc{T}_2$ such that $(\mc{T}_2/I\mc{T}_2)/A \cong \ROI/\varpi \cong \F$.  Thus we have that the image of $\mf{c}_2$ in $\coh^1(\absgal, T\otimes_{\ROI} \F)$ is zero under the composite
\begin{equation*}
\coh^1(\absgal, T \otimes_{\ROI} \mc{T}_2/I\mc{T}_2) \overset{\phi}{\rightarrow} \coh^1(\absgal, T\otimes_{\ROI} ((\mc{T}_2/I\mc{T}_2)/\ker f)) \rightarrow  \coh^1(\absgal, T\otimes_{\ROI} \F).
\end{equation*} 

We now consider $\mf{c}_1 \in \coh^1(\absgal, \Hom_{\T/I}(M_3 \otimes_{\ROI} \T/I, M_1 \otimes_{\ROI} \mc{T}_1/I\mc{T}_1))$.  Let $T' = \Hom_{\ROI}(M_3, M_1)$.  Choose an $\ROI$-module $B\subset \mc{T}_1/I\mc{T}_1$ so that $(\mc{T}_1/I\mc{T}_1)/B \cong \F$.  The fact that $\mf{c}_2$ vanishes in $\coh^1(\absgal, T\otimes_{\ROI} \F)$ gives that $T' \otimes_{\ROI} \F \cong \F(-1)$ where we write $\F(-1)$ to indicate the finite field $\F$ with a $\absgal$-action given by $\omega^{-1}$.  Thus, if $\mf{c}_1$ is nonzero in $\coh^1(\absgal, T' \otimes_{\ROI} \F)$, i.e., is nonzero in $\coh^1(\absgal, \F(-1))$, we obtain, reasoning as above, an element in $\Sel(\F(-1))$.  However, such an element gives a non-trivial finite unramified abelian $p$-extension $K/\rat(\mu_{p})$ with the action of $\Gal(K/\rat)$ on $\Gal(K/\rat(\mu_{p}))$ given by $\omega^{-1}$.  Thus, we obtain a non-trivial subgroup of the $\omega^{-1}$-isotypical piece of the $p$-part of the class group of $\rat(\mu_{p})$. However, Herbrand's theorem says that this implies $p \mid \text{B}_2 = \frac{1}{30}$, clearly a contradiction.  Thus, it must be that $\mf{c}_2$ is zero in $\coh^1(\absgal, T' \otimes_{\ROI} \F)$.  (See \cite{jimbrown} pages 316-317 for the details of the argument showing $\mf{c}_2$ must be zero.)  However, this gives that the exact sequence
\begin{equation*}
0 \rightarrow (M_1 \otimes_{\ROI} \F) \oplus (M_2 \otimes_{\ROI} \F) \rightarrow (\lat/I\lat)/\lat' \rightarrow M_3 \otimes_{\ROI} \F \rightarrow 0
\end{equation*}
splits as a sequence of $\T[\absgal]$-modules where $\lat' = \varpi \lat + (M_1 \otimes_{\ROI} B) + (M_2 \otimes_{\ROI} B))$.  This contradicts Corollary \ref{corl:lattices} (2).  Thus, we have the result that $(\ker \iota_2)^{\vee} = 0$. 
\end{proof}

We are now able to finish the proof of Theorem \ref{thm:selmerbound} in the case where $\Hom_{\T/I}(M_1 \otimes_{\ROI} \mc{T}_1/I\mc{T}_1, M_2 \otimes_{\ROI} \mc{T}_2/I\mc{T}_2)  = 0$.  

\begin{proof} (of Theorem \ref{thm:selmerbound}) Lemma \ref{lem:917} immediately implies that we have the bound
\begin{equation*}
\ord_{p}(\#S(W)) \geq \ord_{p}(\# (\IM \iota_2)^{\vee}).
\end{equation*}
Similarly, we use Lemma \ref{lem:918} to conclude that 
\begin{equation*}
\ord_{p}(\# (\IM \iota_2)^{\vee}) = \ord_{p}(\# \Hom_{\ROI}(\mc{T}_2/I\mc{T}_2, E/\ROI) ^{\vee}).
\end{equation*}
Applying (\cite{hida}, page 98) to our situation we obtain 
\begin{equation*}
\Hom_{\ROI}(\mc{T}_2/I\mc{T}_2, E/\ROI)^{\vee} \cong (\mc{T}_2/I\mc{T}_2)^{\vee \vee}  \cong \mc{T}_2/I\mc{T}_2.
\end{equation*}
This allows us to conclude that 
\begin{equation*}
\ord_{p}(\# (\IM \iota_2)^{\vee}) = \ord_{p} (\# \mc{T}_2/I\mc{T}_2).
\end{equation*}
Corollary \ref{corl:lattices} gives that $\Fitt_{\T}(\mc{T}_2) = 0$ and so $\Fitt_{\T}(\mc{T}_2 \otimes_{\T} \T/I) \subset I$.  Thus we have 
\begin{equation*}
\ord_{p}(\#(\mc{T}_2 \otimes_{\T} \T/I)) \geq \ord_{p}(\#\T/I).
\end{equation*}
However, as $\ord_{p}(\# \mc{T}_2/I\mc{T}_2) = \ord_{p}(\# (\mc{T}_2 \otimes_{\T} \T/I))$, we are able to combine these results to obtain 
\begin{equation*}
\ord_{p}(\# S(W)) \geq \ord_{p} (\# \T/I)
\end{equation*}
as desired. 
\end{proof}

It now only remains to deal with the case where $\Hom_{\T/I}(M_2 \otimes_{\ROI} \mc{T}_2/I\mc{T}_2, M_1 \otimes_{\ROI} \mc{T}_1/I\mc{T}_1) = 0$.  We will see this argument goes much as before so we will be able to reference the first case for most of the work.  Again consider the cocycle 
\begin{equation*}
\mf{c}_1 \in \coh^1(\absgal, \Hom_{\T/I}(M_3 \otimes_{\ROI} \T/I, M_1 \otimes_{\ROI} \mc{T}_1/I \mc{T}_1).
\end{equation*}
As before, we have 
\begin{equation*}
\Hom_{\T/I}(M_3 \otimes_{\ROI} \T/I, M_1 \otimes_{\ROI} \mc{T}_1/I \mc{T}_1) \cong \Hom_{\ROI}(M_3, M_1) \otimes_{\ROI} \mc{T}_1/I\mc{T}_1
\end{equation*}
and so $\mf{c}_1$ can be regarded as a cocycle in $\coh^1(\absgal, \Hom_{\ROI}(M_3, M_1) \otimes_{\ROI} \mc{T}_1/I\mc{T}_1)$. Define a map 
\begin{align*}
\iota_1 : \Hom_{\ROI}(\mc{T}_1/I\mc{T}_1, E/\ROI) &\rightarrow \coh^1(\absgal, \Hom_{\ROI}(M_3, M_1) \otimes_{\ROI} E/\ROI)\\
    f &\mapsto (1\otimes f)(\mf{c}_1).
\end{align*}
Our assumption that $\Hom_{\T/I}(M_2 \otimes_{\ROI} \mc{T}_2/I\mc{T}_2, M_1 \otimes_{\ROI} \mc{T}_1/I\mc{T}_1)  = 0$ give that we can choose $T' = \Hom_{\ROI}(M_3, M_1) \cong \ROI(-1)$ and we have $W' = \Hom_{\ROI}(M_3, M_1) \otimes_{\ROI} E/\ROI \cong (E/\ROI)(-1)$.  As above in Lemma \ref{lem:917} we obtain that $\IM(\iota_1) \subset \Sel((E/\ROI)(-1))$.  However, as was used in the proof of Lemma \ref{lem:918}, $\Sel((E/\ROI)(-1)) = 0$.  Thus, the image of $\iota_1$ must be zero.  This puts us back into the situation where $T = \Hom_{\ROI}(M_3, M_2)$ and $W = \Hom_{\ROI}(M_3, M_2) \otimes_{\ROI} E/\ROI$.  We now are exactly in the situation we were in before and so the same arguments apply and give us the proof of Theorem \ref{thm:selmerbound}. 

\bibliographystyle{alpha}

\end{document}